\documentclass[a4paper,11pt]{article}

\usepackage[T1]{fontenc}
\usepackage{lmodern}

\usepackage{dsfont}

\usepackage[latin1]{inputenc}
\usepackage{amsmath}
\usepackage{amsthm}
\usepackage{amssymb}
\usepackage{mathrsfs}
\usepackage{graphicx}
\usepackage[all]{xy}
\usepackage{hyperref}

\usepackage{makeidx}

\usepackage{stmaryrd}
\usepackage{caption}

\usepackage{abstract} 

\usepackage{cleveref}

\usepackage{xcolor}

\usepackage[toc,page]{appendix}

\newtheorem{thm}{Theorem}[section]
\newtheorem{cor}[thm]{Corollary}
\newtheorem{claim}[thm]{Claim}

\newtheorem{lemma}[thm]{Lemma}
\newtheorem{prop}[thm]{Proposition}

\theoremstyle{definition}
\newtheorem{definition}[thm]{Definition}
\newtheorem{ex}[thm]{Example}
\newtheorem{remark}[thm]{Remark}
\newtheorem{question}[thm]{Question}

\newtheorem{conj}[thm]{Conjecture}

\definecolor{oussamacomment}{rgb}{0.8,0.33,0}

\def\rquotient#1#2{%
	\makeatletter
	\raise.3ex\hbox{$#1$}/\lower.3ex\hbox{$#2$}%
	\makeatother
}	

\makeatletter
\newcommand{\subjclass}[2][2010]{%
	\let\@oldtitle\@title%
	\gdef\@title{\@oldtitle\footnotetext{#1 \emph{Mathematics subject classification.} #2}}%
}
\newcommand{\keywords}[1]{%
	\let\@@oldtitle\@title%
	\gdef\@title{\@@oldtitle\footnotetext{\emph{Key words and phrases.} #1.}}%
}
\makeatother

\newcommand{\Address}{{
		\bigskip
		\small
		
\noindent\textsc{UCLouvain\\ 
Institut de recherche en Math\'ematiques et physique\\
 Chemin du Cyclotron 2\\
1348 Louvain-la-Neuve (Belgium)}\par\nopagebreak
\noindent\textit{E-mail address}: \texttt{oussama.bensaid@uclouvain.be}
  
  	\bigskip
		\small
		
\noindent\textsc{University of Montpellier\\ 
Institut Math\'ematiques Alexander Grothendieck\\
Place Eug\`ene Bataillon\\
34090 Montpellier (France)}\par\nopagebreak
\noindent\textit{E-mail address}: \texttt{anthony.genevois@umontpellier.fr}
  
    \bigskip
		\small
		
\noindent\textsc{University of Paris-Cit\'e\\ 
Institut de Math\'ematiques de Jussieu-Paris Rive Gauche\\
Place Aur\'elie Nemours\\
75013 Paris (France)}\par\nopagebreak
\noindent\textit{E-mail address}: \texttt{romain.tessera@imj-prg.fr}

}}

\makeindex

\title{Coarse separation and splittings in right-angled Artin groups}
\date{\today}
\author{Oussama Bensaid, Anthony Genevois, and Romain Tessera}

\subjclass{Primary 20F65. Secondary 20F67, 20F69.}
\keywords{Right-angled Artin groups, splittings, coarse separation}

\begin{document}

\maketitle

\begin{abstract}
In this article, we characterise geometrically when a right-angled Artin group splits over an abelian subgroup. More precisely, given a finite graph $\Gamma$, we show that $A(\Gamma)$ splits over an abelian subgroup if and only if it is coarsely separable by a family of subexponential growth, which amounts to saying that $\Gamma$ is complete or separated by a complete subgraph. 
\end{abstract}

\tableofcontents

\section{Introduction}
Separation phenomena in groups have long been central in geometric group theory, as they frequently signal the existence of algebraic splittings. A foundational result is Stallings' theorem \cite{MR0415622}: a finitely generated group has more than one end if and only if it splits over a finite subgroup. Equivalently, the Cayley graph of such a group can be divided into at least two deep components by removing a finite set precisely when the group admits a splitting over a finite subgroup. Thus, a coarse separation property already reveals a nontrivial splitting, and therefore an algebraic feature that remains invariant under quasi-isometry.
This perspective has been extended in several directions. Dunwoody and Swenson \cite{MR1760752} provided an initial generalization by broadening the subgroup-theoretic aspect of Stallings' theorem from finite subgroups to finitely generated virtually polycyclic codimension one subgroups. In particular, if $G$ is one-ended and not virtually a surface group, then $G$ splits over a two-ended subgroup if and only if it contains an infinite cyclic codimension-one subgroup. Here, a subgroup $H \leq G$ is called codimension-one if some (or equivalenlty, any) Schreier graph $\mathrm{Sch}(G,H)$ (with respect to a finite generating set) has at least two ends. We refer the reader to \cite{MR2031877} for more background on the connection between codimension-one subgroups and splittings. A further generalization, due to Papasoglu \cite{MR2153400}, addresses the geometric side of Stallings' theorem by replacing finite separating sets with quasi-lines. He proved that a one-ended finitely presented group that is not commensurable to a surface group splits over a two-ended subgroup if and only if its Cayley graph is separated by a quasi-line. In particular, the existence of such a splitting is invariant under quasi-isometry.

\medskip \noindent
In the same spirit, and building on our previous work \cite{bensaid2024coarse}, we investigate coarse separation by families of subsets rather than individual ones. Our approach is quantitative, focusing on the volume growth of the separating sets instead of their quasi-isometry type. We refer to \S \ref{sec:prelim} for our definition of coarse separation.

\medskip \noindent
In the present paper, we consider the class of right angled Artin groups.  Recall that, given a finite simple graph $\Gamma$ (that is, a graph with no loops nor multiple edges), the associated right-angled Artin group $A(\Gamma)$ is defined by the following presentation: its generators correspond to the vertices of $\Gamma$, and two generators commute if and only if the corresponding vertices are connected by an edge. Our main result is the following.

\begin{thm}\label{thm:BigIntro}
Let $\Gamma$ be a finite graph. The following assertions are equivalent:
\begin{itemize}
	\item[(i)] $A(\Gamma)$ is coarsely separable by a family of subexponential growth;
	\item[(ii)] $A(\Gamma)$ splits over an abelian subgroup;
	\item[(iii)] $\Gamma$ is complete or separated by a complete subgraph.
\end{itemize}
\end{thm}

\medskip \noindent
The classification of right-angled Artin groups up to quasi-isometry is a longstanding problem that has only been resolved in certain special cases. It is clear though that right-angled Artin groups both exhibit flexibility and rigidity features depending on the properties of the defining graph. 
For instance, on the one hand, Behrstock and Neumann \cite{MR2376814} showed that all groups $A(\Gamma)$, where $\Gamma$ is a tree of diameter at least 3, are quasi-isometric (see \cite{MR2727658} for a higher-dimensional analogue of this result). On the other hand, widely generalising an earlier result of Bestvina, Kleiner and Sagev \cite{MR2421136}, Huang \cite{MR3692971} proved that if $A(\Gamma_1)$ and $A(\Gamma_2)$ both have finite automorphism groups, then they are quasi-isometric if and only if they are isomorphic. The property of having a finite automorphism group can be detected directly from the graph $\Gamma$ (see \cite{MR3692971}). In particular, such graphs are not separated by complete subgraphs. 
 We have the following non-trivial consequence of Theorem \ref{thm:BigIntro}. 

\begin{cor}
    Let $G$ and $H$ be two quasi-isometric right-angled Artin groups. If $G$ splits over an abelian subgroup, then so does $H$.
\end{cor}


\medskip \noindent
We now present the strategy of the proof of Theorem \ref{thm:BigIntro}. We first present the main ingredients, which are of independent interest, and then explain how to combine them to obtain the final result.

\medskip \noindent
{\bf Ingredient 1: From right-angled Artin groups to graph products of finite groups.} Given a graph $\Gamma$ and a weight $\nu : V(\Gamma) \to \mathbb{N}_{\geq 2}$, we denote by $\Gamma(\nu)$ the graph product $\Gamma \{ \mathbb{Z}_{\nu(u)} \mid u \in V(\Gamma)\}$.

\medskip \noindent
The following theorem plays a central role in the proof of  Theorem \ref{thm:BigIntro}. (See Theorem~\ref{thm:CoarseSepGP} for a more general statement.) 

\begin{thm}\label{thm:IntroRAAGGP}
Let $\Gamma$ be a finite graph that is not a join and that contains at least two vertices. If $A(\Gamma)$ is coarsely separable by a family of subexponential growth, then, for every $k \geq 0$, there exists a weight $\nu \geq k$ such that $\Gamma(\nu)$ is coarsely separable by a family of subexponential growth.  
\end{thm}

\noindent
The assumptions on $\Gamma$ amounts to saying that $A(\Gamma)$ be acylindrically hyperbolic. We show that, in any finitely generated acylindrically hyperbolic group, a coarsely separating family with subexponential growth must coarsely separate a bi-infinite Morse geodesic (see Theorem~\ref{thm:SeparatingAcylHyp}).
We then observe that, in the right-angled Artin group $A(\Gamma)$, every bi-infinite Morse geodesic lies inside a quasi-isometrically embedded copy of $\Gamma(\nu)$, for some weight $\nu$ whose minimum is sufficiently large compared to the Morse gauge of the geodesic (see Theorem~\ref{thm:MorseSubBuilding}). This reduces the problem of coarse separation in a right-angled Artin group to that of coarse separation in a graph product of finite groups.

\medskip \noindent
{\bf Ingredient 2: Coarse separation in hyperbolic groups.} 
Despite the fact that right-angled Artin groups are usually not hyperbolic, many graph products of finite groups are hyperbolic, which will allow us to deduce some valuable information on coarse separation in right-angled Artin groups thanks to Theorem~\ref{thm:IntroRAAGGP}.
The hyperbolic case follows from the main result of a companion paper \cite{BGT26hyp}, where we prove the following result.

\begin{thm}\label{thm:IntroHyp}
Let $G$ be a hyperbolic group. Assume that $G$ is one-ended and not virtually a surface group. Then $G$ is coarsely separable by a family of subexponential growth if and only if it splits over a virtually cyclic subgroup. 
\end{thm} 

\noindent
Combining \Cref{thm:IntroHyp} and \Cref{thm:IntroRAAGGP} we deduce the following special case of Theorem \ref{thm:BigIntro}.

\begin{cor}\label{cor:RAAGsquareFree}
Let $\Gamma$ be a finite graph. If $\Gamma$ is not complete, without cut-point, and $\square$-free, then $A(\Gamma)$ is not coarsely separable by a family of subexponential growth. 
\end{cor}

\medskip
\noindent
{\bf Ingredient 3: Thickness and combinatorial reduction.} Roughly speaking, a space $X$ is \emph{thick relative to a pair $(\mathcal{A}, \mathcal{B})$} of two collections of subspaces if any two points of $X$ can be connected by a chain of subspaces belonging to $\mathcal{A}$ such that the intersection between any two consecutive subspace along this chain belong to $\mathcal{B}$. A rather elementary but useful observation is:

\begin{prop}\label{prop:ThickNotSeparable}
Let $X$ be an $(\mathcal{A}, \mathcal{B})$-thick metric space. Assume that the subspaces in $\mathcal{A}$ are uniformly not coarsely separable by families of subexponential growths and that the subspaces in $\mathcal{B}$ have exponential growth.  Then $X$ is not coarsely separable by a family of subexponential growth.
\end{prop}
\noindent It is worth mentioning that Proposition~\ref{prop:ThickNotSeparable} has applications to others groups than right-angled Artin groups, such as mapping class groups of surfaces, free products with amalgamation, or certain connected solvable Lie groups. See Section~\ref{sec:thickAgain}.

\medskip
\noindent
In \S \ref{sec:unpinched_graphs}, we establish a graph-theoretic decomposition for triangle-free graphs that are not separated by a complete subgraph. As a consequence, \Cref{thm:RAAGs_are_thick} shows that, if $\Gamma$ is triangle-free and not separated by a complete subgraph, then $A(\Gamma)$ is $(\mathcal{A},\mathcal{B})$-thick, where $\mathcal{A}$ consists of uniform neighborhoods of left cosets of parabolic subgroups associated to cycles in $\Gamma$, and $\mathcal{B}$ consists of left cosets of parabolic subgroups associated to pairs of non-adjacent vertices. This reduces the proof of Theorem~\ref{thm:BigIntro} to the case where $\Gamma$ is a cycle of length at least $4$, and provides the main technical application of \cref{prop:ThickNotSeparable} in the setting of right-angled Artin groups: under suitable assumptions, if a family of subexponential growth coarsely separates an amalgamated product, then it also coarsely separates one of the factor groups. Finally, in \S \ref{sec:trianglefreeReduc}, we show that it is enough to restrict to right-angled Artin groups defined by triangle-free graphs.

\medskip \noindent
{\bf Recipe.} If $\Gamma$ is complete (and non-empty), then $A(\Gamma)$ is free abelian and is clearly coarsely separable by a subspace of polynomial growth. If $\Gamma$ contains a separating complete subgraph, then $A(\Gamma)$ splits over a free abelian subgroup, and consequently is coarsely separable by a subspace of polynomial growth. Conversely, assuming that $\Gamma$ is not complete and cannot be separated by a complete subgraph, we need to show that $A(\Gamma)$ is not coarsely separable by a family of subexponential growth. By ingredient 3, it suffices to verify that $A(C_n)$ cannot be coarsely separated by a family of subexponential growth for every $n \geq 4$. For $n=4$, $A(C_n)$ is a product of two free group of rank two, so the desired conclusion follows from \cite[Theorem 1.3]{bensaid2024coarse}. For $n \geq 5$, the desired conclusion follows from Corollary~\ref{cor:RAAGsquareFree}.

\paragraph{Discussion.} Beyond its conceptual interest, and the new quasi-isometric invariants it provides, it is worth mentioning that Theorem~\ref{thm:BigIntro} imposes restrictions on coarse embedding between right-angled Artin groups. Roughly speaking, it implies that, with respect to some JSJ decompositions above abelian subgroups, rigid pieces are sent into neighbourhoods of rigid pieces. See Section~\ref{section:JSJ} for a precise statement and for an explicit example for which we can show the non-existence of coarse embeddings. This connection between JSJ decomposition and coarse embeddings remains to be further investigated. (See \cite{MR4186478} for some information on the cyclic case.)

\medskip \noindent
As another application, we will use Theorem~\ref{thm:BigIntro} in \cite{CodimRAAG} in order to characterise algebraically which right-angled Artin groups virtually split over $\mathbb{Z}^n$, for a fixed $n \geq 0$, solving a problem from \cite{MR3954281}. 

\medskip \noindent
In view of Theorem~\ref{thm:BigIntro}, two questions can be naturally asked. Right-angled Artin groups belong to two more general family of groups: graph products of groups and cocompact special groups. On the one hand, we expect that:

\begin{conj}
Let $\Gamma$ be a finite graph and $\mathcal{G}$ a collection of infinite finitely generated virtually nilpotent groups. The graph product $\Gamma \mathcal{G}$ is coarsely separable by a family of subexponential growth if and only if $\Gamma$ is complete or separated by a complete subgraph.
\end{conj}

\noindent
A proof should follow by combining a generalisation of Theorem~\ref{thm:IntroHyp} to relatively hyperbolic groups (conjectured in \cite{BGT26hyp}) with the thickness of non-relatively hyperbolic graph products that can be deduced from \cite[Theorem~8.35]{QM}. 

\medskip \noindent
On the other hand, we ask:

\begin{question}
When is a cocompact special group (e.g.\ a right-angled Coxeter group, or more generally a graph product of finite groups), coarsely separable by a family of subexponential growth?
\end{question}

\subsection*{Acknowledgements}
The first-named author acknowledges support from the FWO and F.R.S.-FNRS under the Excellence of Science (EOS) programme (project ID 40007542).

\section{Preliminaries}\label{sec:prelim}

\noindent
In this section, we recall some basic definitions and notations related to coarse separation, following \cite{bensaid2024coarse}.

\medskip \noindent
If $(X,d)$ is a metric space, we denote by $B(x,r)$ (resp.\ $S(x,r)$) the closed ball (resp.\ the sphere) of radius $r$, namely:
$$B(x,r) := \{ z \in X \mid d(x,z) \leq r \}, \qquad S(x,r) := \{ z \in X \mid d(x,z) = r \}.$$
If $A \subseteq X$ is a subset, we denote by $|A| \in \mathbb{N} \cup \{+\infty\}$ its cardinality; and by $A^{+r}$ its $r$-neighborhood, namely:
$$A^{+r} := \{ x \in X \mid d(x,A) \leq r \}.$$
Our definition of coarse separation is the following:

\begin{definition}
Let $X$ be a metric space and $\mathcal{A}, \mathcal{Z}$ two collections of subspaces. Then $\mathcal{A}$ is \emph{coarsely separated by $\mathcal{Z}$} if there exist $k,L \geq 0$ such that, for every $D \geq 0$, one can find some $A \in \mathcal{A}$ and $Z \in \mathcal{Z}$ satisfying:
\begin{itemize}
	\item $A$ is $k$-connected, and 
	\item $A \backslash Z^{+L}$ has at least two $k$-coarsely connected components with points at distance $\geq D$ from $Z$. 
\end{itemize}
\end{definition}

\noindent
Notice that, in the case where $X$ is a graph and $\mathcal{Z}$ a collection of subgraphs, a connected subgraph $Y \subset X$ is coarsely separated by $\mathcal{Z}$ if there exists $L \geq 0$ such that, for every $D \geq 0$, there is some $Z \in \mathcal{Z}$ such that $Y \setminus Z^{+L}$ has at least two connected components with points at distance $\geq D$ from $Z$. 

\begin{ex}\label{ex:CoarseSep}
Note that, if $\mathcal{Z}$ coarsely separates some $A \in \mathcal{A}$, then it coarsely separates $\mathcal{A}$. The converse, however, is not true. For instance, if $X = \mathbb{R}$, if $\mathcal{Z} = \{ \{0\} \} $, and if $\mathcal{A}$ is the family of segments $[-n,n]$ with $n \in \mathbb{N}$, then $\mathcal{Z}$ coarsely separates $\mathcal{A}$ but does not coarsely separate any segment $[-n,n]$.
\end{ex}

\noindent
In this paper, we focus our attention to coarse (non-)separation by families with slow growth. More precisely:

\begin{definition}\label{def:growth_separating_family}
Let $X$ be a graph of bounded degree, and let $\mathcal S$ be a family of subgraphs of $X$. We define its relative growth by
$$V_{\mathcal{S}}(R) := \sup_{s\in S,\ S\in \mathcal S}|B_X(s,R)\cap S|.$$
We say that $\mathcal S$ has \emph{exponential volume growth} if
$$\limsup_{r\to\infty} \frac{1}{r}\log V_{\mathcal S}(r) > 0.$$
We denote by $\mathfrak{M}_{\mathrm{exp}}$ the class of bounded degree graphs such that any coarsely separating family of subsets must have exponential volume growth.
\end{definition}

\section{Coarse separation in acylindrically hyperbolic groups}

\noindent
This section is dedicated to the proof of the following statement. We refer the reader to Section~\ref{section:Morse} for the definitions of acylindrically hyperbolic groups and Morse geodesics. Roughly speaking, our theorem claims that, in a finitely generated hyperbolic-like group, a coarsely separating family of subexponential growth must coarsely separate two hyperbolic-like directions. 

\begin{thm}\label{thm:SeparatingAcylHyp}
Let $G$ be a finitely generated acylindrically hyperbolic group and let $\mathcal{Z}$ a family of subspaces. If $\mathcal{Z}$ has subexponential growth and coarsely separates $G$, then there exists a bi-infinite Morse geodesic in $G$ that is coarsely separated by $\mathcal{Z}$. 
\end{thm}

\noindent
Section~\ref{section:Morse} records some basic definitions and properties related to Morse geodesics, and Section~\ref{section:SeparatingMorse} contains the proof of Theorem~\ref{thm:SeparatingAcylHyp}.

\subsection{Morse geodesics}\label{section:Morse}

\noindent
First, let us recall the definition of Morse subspaces. 

\begin{definition}
Let $X$ be a geodesic metric space. Given a map $M : [0,+ \infty) \times [0,+ \infty)$, a subspace $Y \subset X$ is \emph{$M$-Morse}, or \emph{Morse with Morse gauge $M$}, if, for all $A>0$ and $B \geq 0$, every $(A,B)$-quasigeodesic connecting two points of $Y$ remains in the $M(A,B)$-neighbourhood of $Y$. 
\end{definition}

\noindent
It is worth mentioning that a Morse subspace is always coarsely connected and quasi-isometrically embedded. 

\medskip \noindent
Morse subspaces include quasiconvex subspaces in hyperbolic spaces, but they can also be easily found in \emph{acylindrically hyperbolic groups}, which is why they interest us here. Recall that:

\begin{definition}
A group $G$ is \emph{acylindrically hyperbolic} if it admits an action on some hyperbolic space $X$ that is non-elementary and \emph{acylindrical}, i.e.\ for every $D \geq 0$, we can find $N,L \geq 0$ such that
$$\forall x,y \in X, \ d(x,y) \geq L \Rightarrow \# \{ g \in G \mid d(g \cdot x, g \cdot y ) \leq D \} \leq N.$$
\end{definition}

\noindent
The following observation records the fact that, in an acylindrically hyperbolic, there are Morse rooted quasi-trees starting from every point. 

\begin{lemma}\label{lem:MorseQT}
Let $G$ be acylindrically hyperbolic group. There exists a Morse gauge $M$ such that, for every $g \in G$, there exists an $M$-Morse quasi-tree $(T_2,o) \hookrightarrow (G,g)$.
\end{lemma}

\noindent
In our lemma, $(T_2,o)$ denotes a $2$-regular tree rooted at $o$. 

\begin{proof}[Proof of Lemma~\ref{lem:MorseQT}.]
According to \cite[Theorem~2.24]{MR3589159} and \cite[Theorem~2]{MR3519976}, our acylindrically hyperbolic group $G$ contains a Morse free subgroup, hence a Morse rooted tree based at $1$. By translating our Morse tree, by $g$, we get a Morse rooted tree based at $g$, as desired. 
\end{proof}

\subsection{Separating Morse geodesics}\label{section:SeparatingMorse}

\noindent
In this section, our goal is to prove Theorem~\ref{thm:SeparatingAcylHyp}. We start by proving the following general observation: 

\begin{prop}\label{prop:RayInComponent}
Let $\sigma$ be a subexponential map, $\rho_-,\rho_+: [0,\infty) \to [0,\infty)$ two increasing functions, and let $N$ be a constant. There exists some $Q \geq 0$ such that the following holds. Let $X$ be a graph of degree $\leq N$, $Z \subset V(X)$ a subset of growth $\leq \sigma$, and $\varphi : (T_2,o) \hookrightarrow (X,x)$ a coarse embedding with control functions $\rho_-$ and $\rho_+$. If $d(x,Z) \geq Q$, then there exists a geodesic ray in $(	T_2,o)$, starting at $o$, such that $\varphi(\rho)$ is disjoint from $Z$ and not contained in any neighbourhood of $Z$. 
\end{prop}

\noindent
Its proof will be based on the following preliminary observation:

\begin{lemma}\label{lem:SubExpTree}
Let $(T,o)$ be a $2$-regular rooted tree and $S \subset V(T)$ a subset whose growth based at $o$ is subexponentiel. There exists a geodesic ray $\rho$ starting at $o$ such that $\rho \cap S \subset B(o,r_0)$ where $r_0 \geq 0$ is chosen so that $\sum_{k \geq r_0} \sigma_S(k)/2^k \leq 1$. 
\end{lemma}

\noindent
Here, $\sigma_S(k)$ refers to $\# \{ s \in S \mid d(o,s)=k\}$, i.e.\ $\sigma_S$ is the \emph{spherical growth} of $S$ in $(T,o)$.

\begin{proof}[Proof of Lemma~\ref{lem:SubExpTree}.]
Let $T^-$ denote the tree obtained from $T$ by removing the bottom tree hanging at each point of $S$ outside $B(o,r_0)$. At the level $r_0$, we keep the $x_0:= 2^{r_0}$ vertices of $T$ in $T^-$. At the next level, the number of vertices doubles to $2x_0$, and we remove $\leq \sigma_S(r_0+1)$ vertices to get the level of $T^-$. So the level of $T^-$ contains at most $\leq x_1:= 2x_0 - \sigma_S(r_0+1)$ vertices. By iterating, we find that the $(r_0+k)$th level of $T^-$ contains at most $x_k$ vertices where the sequence $(x_n)_n$ is defined by
$$\left\{ \begin{array}{l} x_0=2^{r_0} \\ x_{n+1}= 2x_n - \sigma_S(r_0+k) \end{array} \right..$$
Our goal is to prove that $x_n \geq 2^{n-1}$ for every $n \geq 0$. This will show that $T^-$ is a subtree of exponential growth. It particular, it contains an infinite ray $\rho$ starting from $o$. By construction, $\rho$ is disjoint from $S$ outside $B(o,r_0)$.

\medskip \noindent
Setting $y_n:=x_n/2^n$, we get a new sequence defined by
$$\left\{ \begin{array}{l} y_0= 2^{r_0} \\ y_{n+1}= y_n - \sigma_S(r_0+n)/2^{n+1} \end{array} \right..$$
We have
$$\begin{array}{lcl} y_n & = & \displaystyle 2^{r_0}- \sum\limits_{k <n} \frac{\sigma_S(r_0+k)}{2^{k+1}} \geq 2^{r_0}- \sum\limits_{k \geq 0 } \frac{\sigma_S(r_0+k)}{2^{k+1}}  \\ \\ & \geq & \displaystyle 2^{r_0} - 2^{r_0-1} \sum\limits_{k \geq 0} \frac{\sigma_S(r_0+k)}{2^{r_0+k}} = 2^{r_0} - 2^{r_0-1} \sum\limits_{k \geq r_0} \frac{\sigma_S(k)}{2^{k}}  \\ \\ & \geq & \displaystyle 2^{r_0} - 2^{r_0-1} = 2^{r_0-1} \geq 1/2 \end{array}$$
by definition of $r_0$. We conclude that $x_n = 2^ny_n \geq 2^{n-1}$ for every $n \geq 0$, as desired. 
\end{proof}

\begin{proof}[Proof of Proposition~\ref{prop:RayInComponent}.]
For convenience, we write $|p|= d(x,p)$ for every $p \in V(X)$. We also consider the thickning
$$Z^+:= \bigcup\limits_{z \in Z} B(z, R(|z|)),$$
where $R$ is a map we have to chose carefully. More precisely, we want the following conditions to hold:
\begin{itemize}
	\item $R$ is increasing and unbounded;
	\item $R$ increases sufficiently slowly so that $Z^+$ has subexponential growth based at $x$;
	\item $R$ increases sufficiently slowly so that $R(a) \leq a/2$ for every $a \geq 0$.
\end{itemize}
Notice that the spherical growth of $Z^+$ at $x$ satisfies
$$\sigma_{Z^+}(\ell) \leq \sigma_Z(\ell) \cdot \beta_X(R(\ell))$$
where $\eta_X$ denotes the growth function of $X$. Thus, by choosing $R$ of the form $\log \circ \log$, the three conditions above are satisfied. 

\medskip \noindent
According to Lemma~\ref{lem:SubExpTree}, there exists a geodesic ray $\rho$ in $T_2$ starting at $o$ that may intersect $\varphi^{-1}(Z^+)$ only in the ball $B(o,r_0)$, where $r_0$ is a constant depending only on the growth of $Z$, the degree $N$, and the functions $\rho_-,\rho_+$. Notice that, if $z \in Z$ is such that the closest point of $Z^+$ to $x$ belongs to $B(z, R(|z|))$, then we have
$$|z|  = d(x,z) \leq d(x,Z^+) + R(|z|) \leq d(x,Z^+) + |z|/2,$$
hence 
$$d(x,Z^+) \geq |z|/2 = d(x,z)/2 \geq d(x,Z) /2.$$
Thus, if we choose $Q$ sufficiently large (more precisely, if $d(x,Z)$ is assumed to be sufficiently large compared to $\rho_+(r_0)$), then $\varphi^{-1}(Z^+)$ must be disjoint from $B(o,r_0)$, which implies that $\varphi(\rho)$ is disjoint from $Z^+$. 

\medskip \noindent
A fortiori, $\varphi(\rho)$ is disjoint from $Z$, as desired. Moreover, since $R$ is unbounded, $\varphi(\rho)$ being disjoint from $Z^+$ implies that $\varphi(\rho)$ is not contained in any neighbourhood of $Z$, concluding the proof. 
\end{proof}

\begin{proof}[Proof of Theorem~\ref{thm:SeparatingAcylHyp}.]
Because $\mathcal{Z}$ coarsely separates $G$, there exist $k,L \geq 0$ such that, for every $D \geq 0$, we can find a $k$-connected $Z \in \mathcal{Z}$ such that $G \backslash Z^{+L}$ contains at least two $k$-connected components of with points at distance $\geq D$ from $Z^{+L}$. Let $Z \in \mathcal{Z}$ be such a subspace given for some $D$ larger that the constant given by Proposition~\ref{prop:RayInComponent} (where $\sigma$ is the growth of $\mathcal{Z}$, $N$ is the degree of (a Cayley graph of) $G$, and where $\rho_-,\rho_+$ are the affine functions such that an $M$-Morse quasi-tree in $G$  is $(\rho_-,\rho_+)$-quasi-isometrically embedded). Fix two points $x,y \in G$ that lie in distinct $k$-connected components of $G \backslash Z^{+L}$ and at distance $\geq D$ from $Z^{+L}$. 

\medskip \noindent
According to Lemma~\ref{lem:MorseQT}, there exists an $M$-Morse quasi-tree $(T_2,o) \to (G,x)$. Since we have chosen $D$ sufficiently large, it follows from Proposition~\ref{prop:RayInComponent} that our quasi-tree contains a geodesic ray $\zeta$, starting from $x$, disjoint from $Z^{+L}$, and not contained in any neighbourhood of $Z$. Similarly, we know that there exists an $M$-Morse geodesic ray $\xi$ starting from $y$, disjoint from $Z^{+L}$, and not contained in any neighbourhood of $Z$. Since the Morse boundary satisfy the visibility property (see for instance \cite[Proposition~3.11]{MR3737283}), there exists some Morse geodesic $\eta$ with two subrays at finite Hausdorff distance from $\zeta$ and $\xi$. Then, $\eta$ contains two subrays contained in distinct $k$-connected component of $G \backslash Z^{+L}$ but not contained in any neighbourhood of $Z$. Thus, $\eta$ is coarsely separated by $Z$, and a fortiori by $\mathcal{Z}$. 
\end{proof}

\section{Coarse separation in thick spaces}

\noindent
In this section, we introduce a notion \emph{thickness} which we will use as an obstruction to coarse separation, thanks to Propositions~\ref{prop:Thick} and~\ref{prop:coarse_sep_family_converse} below. It is reminiscent of \emph{thick metric spaces} introduced in \cite{MR2501302} as an obstruction to relative hyperbolicity. 

\begin{definition}
Let $X$ be a metric space and $\mathcal{A}, \mathcal{B}$ two collections of subspaces. A subset $E \subset X$ is \emph{$(\mathcal{A},\mathcal{B})$-thick} if, for all $x,y \in E$, there exist $A_1, \ldots, A_n \in \mathcal{A}$ such that $x \in A_1$, $y \in A_n$, and $A_i \cap A_{i+1}$ contains some $B \in \mathcal{B}$ for every $1 \leq i \leq n-1$.
\end{definition}

\noindent
In order to illustrate our definition, let us mention a couple of motivating examples.

\begin{ex}\label{ex:thick_spaces}
Thick metric spaces include for instance:
\begin{itemize}
        \item[(1)] Let $G = A_1 *_B A_2$ be an amalgamated product, and let $X$ denote a Cayley graph of $G$ (with respect to a finite generating set). Let $\mathcal{A}$ be the family of left cosets of $A_1$ and of $A_2$, and let $\mathcal{B}$ be the family of left cosets of $B$ (viewed as subgraphs of $X$). Then $X$ is $(\mathcal{A},\mathcal{B})$-thick.
        \item[(2)] Let $X$ and $Y$ be connected graphs, and let $r \geq 1$ be an integer. Let $\mathcal{A}$ be the collection of subgraphs $X \times B_Y(y,r)$, where $y$ ranges over $V(Y)$; and let $\mathcal{B}$ be the collection of subgraphs $X \times \{y\}$, where $y$ ranges over $V(Y)$. Then $X \times Y$ is $(\mathcal{A},\mathcal{B})$-thick.
\end{itemize}
\end{ex}

\noindent
The idea of the second example generalizes to left cosets of a subgroup as the following:

\begin{lemma}\label{lem:thick_family_left_cosets}
Let $G$ be a group endowed with the word metric associated to a generating set $S$, and let $H \leq G$ be a subgroup. Set
$$\mathcal{A}=\{(gH)^{+1} \mid g \in G \}.$$
For each $s \in S$, set $ K_s := H \cap sHs^{-1}$, and let
$$\mathcal{B}=\{gK_s \mid g \in G, s \in S\}.$$
Then $G$ is $(\mathcal{A},\mathcal{B})$-thick.
\end{lemma}

\begin{proof}
First, notice that for every $g \in G$ and $s \in S$, $g K_s \subset gs H s^{-1} \subset gs H^{+1}$. Therefore, $g K_s \subset g H^{+1} \cap gs H^{+1}$. Now, let $x,y \in G$ and write $y=xs_1\cdots s_n$ with $s_i \in S^{\pm 1}$. Set $x_i=xs_1\cdots s_i$ and $A_i=(x_iH)^{+1} \in \mathcal{A}$. Then $x \in A_0$ and $y \in A_n$. For any $0 \leq i \leq n-1$, $A_i \cap A_{i+1}$ contains $x_i K_{s_{i+1}} \in \mathcal B$.
\end{proof}

\noindent
Before proving the two centreal propositions of our section, let us record a couple of observations regarding thickness. Our first statement shows that thickness is naturally transitive. 

\begin{lemma}\label{lem:thickness_transitivity}
Let $X$ be a metric space and let $\mathcal{A}, \mathcal{B}, \mathcal{C}, \mathcal{D}$ be four collections of subspaces. Suppose that $X$ is $(\mathcal{A},\mathcal{B})$-thick and that every $A \in \mathcal{A}$ is $(\mathcal{C},\mathcal{D})$-thick. Assume moreover that, whenever $C,C' \in \mathcal{C}$ intersect, the intersection $C \cap C'$ contains some element of $\mathcal{B} \cup \mathcal{D}$. Then $X$ is $(\mathcal{C},\mathcal{B}\cup \mathcal{D})$-thick.
\end{lemma}

\begin{proof}
Let $x,y \in X$. Since $X$ is $(\mathcal{A},\mathcal{B})$-thick, there exist $A_1,\dots,A_n \in \mathcal{A}$ such that $x \in A_1$, $y \in A_n$, and for every $1 \leq i \leq n-1$ the intersection $A_i \cap A_{i+1}$ contains some $B_i \in \mathcal{B}$. For each $i$, choose a point $p_i \in B_i$, and set $p_0=x$ and $p_n=y$.

\medskip \noindent
Now, we move from $p_{i-1}$ to $p_i$ in $A_i$ using the $(\mathcal{C},\mathcal{D})$-thickness: for every $1 \leq i \leq n$, there exist $C_{i,1},\dots,C_{i,m_i} \in \mathcal{C}$ such that $p_{i-1} \in C_{i,1}$, $p_i \in C_{i,m_i}$, and for every $j$ the intersection $C_{i,j} \cap C_{i,j+1}$ contains some element of $\mathcal{D}$. Concatenating these sequences, we obtain
$$(C_{1,1},\dots,C_{1,m_1}),\ (C_{2,1},\dots,C_{2,m_2}),\ \dots,\ (C_{n,1},\dots,C_{n,m_n}).$$
Inside each block, a consecutive intersection contains an element of $\mathcal{D}$. Moreover, since $p_i$ is contained in both $C_{i,m_i}$ and $C_{i+1,1}$, they intersect and $C_{i,m_i} \cap C_{i+1,1}$ contains some element of $\mathcal{B} \cup \mathcal{D}$. 
\end{proof}

\noindent
The following is an easy observation, whose proof is left to the reader.

\begin{lemma}\label{lem:thick_neighborhood}
Let $X$ be a metric space, let $\mathcal{A},\mathcal{B}$ be two collections of subspaces, and let $E \subset X$ be $(\mathcal{A},\mathcal{B})$-thick. Then, for every $r \geq 0$, the neighborhood $E^{+r}$ is $(\mathcal{A}^{+r},\mathcal{B})$-thick, where $\mathcal{A}^{+r}=\{A^{+r} \mid A \in \mathcal{A}\}$.\qed
\end{lemma}

\noindent
We now turn to the main results of this section. 

\begin{prop}\label{prop:Thick}
Let $X$ be a coarsely connected space and let $\mathcal{A}, \mathcal{B}, \mathcal{Z}$ be three collections of subspaces. Assume that $X$ is $(\mathcal{A},\mathcal{B})$-thick, and that every subspace in $\mathcal{A}$ is coarsely $k$-connected in $X$, for some $k \geq 0$. If $\mathcal{Z}$ coarsely separates $X$, then one of the following assertions holds:
\begin{itemize}
    \item[(1)] $\mathcal{Z}$ coarsely separates $\mathcal{A}$; or
    \item[(2)] there exists $Z \in \mathcal{Z}$ that coarsely contains some $B \in \mathcal{B}$.
\end{itemize}
\end{prop}

\begin{proof}
Suppose that $(1)$ does not hold. On the one hand, because $\mathcal{Z}$ coarsely separates $X$, there exists $L \geq 0$ such that the following holds: for every $D \geq 0$, there exist $Z \in \mathcal{Z}$ and $x,y \in X$ such that $x$ and $y$ lie at distance $>D$ from $Z$ and belong to two distinct $k$-connected components of $X \backslash Z^{+L}$. On the other hand, because $\mathcal{Z}$ does not coarsely separate $\mathcal{A}$, there exists $D = D(k,L) \geq 0$ such that the following holds: for every $A \in \mathcal{A}$ and every $Z \in \mathcal{Z}$, at most one $k$-connected components of $A \backslash Z^{+L}$ is not contained in $Z^{+D}$. Now take $Z \in \mathcal{Z}$ and $x,y \in X$ such that  $x$ and $y$ lie at distance $>D$ from $Z$ and belong to two distinct $k$-connected component of $X \backslash Z^{+L}$. 

\medskip \noindent
Next, because $X$ is $(\mathcal{A}, \mathcal{B})$-thick, we can find $A_1, \ldots, A_n \in \mathcal{A}$ such that $x \in A_1$, $y \in A_n$, and $A_i \cap A_{i+1} \in \mathcal{B}$ for every $1 \leq i \leq n-1$. It follows from the previous paragraph that, for every $1 \leq i \leq n$, there exists at most one $k$-connected component of $A_i \backslash Z^{+L}$ that is not contained in $Z^{+D}$. If no such component exists, then $A_i$ must be contained in $Z^{+D}$, and we are done because $A_i \cap A_{i+1} \in \mathcal{B}$ or $A_i \cap A_{i-1} \in \mathcal{B}$. So, from now on, we assume that there always exists such a component $C_i$. Notice that $A_i \subset C_i \cup Z^{+D}$.

\medskip \noindent
For every $1 \leq i \leq n$, let $C_i^+$ denote the $k$-connected component of $X \backslash Z^{+L}$ containing $C_i$. Notice that, since $x$ lies at distance $>D$ from $Z^{+L}$ and $A_1 \subset C_1 \cup Z^{+D}$, necessarily $x \in C_1$, hence $x \in C_1^+$. Similarly, we have $y \in C_n^+$. We know that $x$ and $y$ must lie in distinct $k$-connected components of $X\backslash Z^{+L}$, so we deduce that $C_1^+ \neq C_n^+$. As a consequence, we can find some $1 \leq i \leq n-1$ such that $C_i^+ \neq C_{i+1}^+$. In fact, as distinct $k$-connected components, we have $C_i^+ \cap C_{i+1}^+ = \emptyset$, which implies that
$$A_i \cap A_{i+1} \subset (C_i^+ \cup Z^{+D}) \cap (C_{i+1}^+ \cup Z^{+D}) \subset Z^{+D}.$$
Thus, the subspace $A_i \cap A_{i+1} \in \mathcal{B}$ is contained in a neighbourhood of $Z$, as desired. 
\end{proof}

\noindent
We emphasize that, despite the fact that, if $\mathcal{Z}$ coarsely separates some $A \in \mathcal{A}$, then it coarsely separates $\mathcal{A}$, the converse may not hold; see for instance Example~\ref{ex:CoarseSep} or \Cref{remark:coarsesep_of_family}. The following lemma gives a condition under which the converse does hold. 

\begin{prop}\label{prop:coarse_sep_family_converse}
Let $X$ be a graph of bounded degree, and let $\mathcal{A}, \mathcal{Z}$ be two families of subgraphs. Assume that $\mathcal{Z}$ has subexponential growth, and that there exist two increasing functions $\rho_-,\rho_+: [0,\infty) \to [0,\infty)$ such that, for every $A \in \mathcal{A}$ and every $a \in A$, there exists a coarse embedding $(T_2,o) \hookrightarrow (A,a)$ with control functions $\rho_-,\rho_+$. If $\mathcal{Z}$ coarsely separates $\mathcal{A}$, then there exist $Z \in \mathcal{Z}$ and $A \in \mathcal{A}$ such that $Z$ coarsely separates $A$. In particular, $\mathcal{Z}$ coarsely separates some $A \in \mathcal{A}$.
\end{prop}

\begin{proof}
On the one hand, $\mathcal{Z}$ coarsely separates $\mathcal{A}$, so there exist $k,L\geq 0$ such that, for every $D \geq 0$, there exist $A \in \mathcal{A}$ and $Z \in \mathcal{Z}$ such that $A \setminus Z^{+L}$ has at least two $k$-coarsely connected components with points at distance $\geq D$ from $Z$.

\medskip \noindent
On the other hand, the family $\mathcal{Z}^{+L}$ has subexponential growth. Let $\sigma$ be its growth function, and let $Q \geq 0$ be the constant given by \Cref{prop:RayInComponent}, which only depends on $\sigma$, the control functions $\rho_-,\rho_+$ and $X$. Let $A \in \mathcal{A}$ and $Z \in \mathcal{Z}$ be such that there exist $a_1,a_2 \in A \setminus Z^{+L}$ in two different $k$-coarsely connected components and at distance $> Q$ from $Z$. By \Cref{prop:RayInComponent}, there exists a coarse embedding $f:(\mathbb{R}_+,0) \hookrightarrow (A,a_1)$ with control functions $\rho_-,\rho_+$ such that $f(\mathbb{R}_+)$ is disjoint from $Z^{+L}$ and not contained in any neighborhood of $Z$. Up to taking $d(\{a_1,a_2\},Z^{+L})$ sufficiently large compared to $\rho_+(1)$, $f(\mathbb{R}_+)$ stays in the same $k$-coarsely connected component containing $a_1$. The same holds for $a_2$. We conclude that $Z$ coarsely separates $A$.
\end{proof}

\begin{remark}\label{remark:coarsesep_of_family}
It should be noted that the condition that $\mathcal{Z}$ has subexponential growth is necessary. Indeed, consider $X = (T_2,o) \times \mathbb{R}$, and set $Z = \bigcup_{n \in \mathbb{N}} S_{T_2}(o,n) \times \{n\}$. Let $\mathcal{A}$ be the family of copies $(T_2,o) \times \{n\}$ indexed by $n \in \mathbb{N}$. It is clear that $Z$ coarsely separates the family $\mathcal{A}$, but does not coarsely separate any copy $(T_2,o) \times \{n\}$.
\end{remark}

\section{Coarse separation in graph products}

\noindent
This section is dedicated to the proof of the following statement:

\begin{thm}\label{thm:CoarseSepGP}
Let $\Gamma$ be a finite graph and $\mathcal{G}$ a collection of infinite finitely generated groups indexed by $V(\Gamma)$. Assume that $\Gamma$ is not a join and contains at least two vertices. If $\Gamma \mathcal{G}$ is coarsely separable by a family of subexponential growth, then, for every $k \geq 0$, there exists a weight $\nu \geq k$ such that $\Gamma(\nu)$ is coarsely separable by a family of subexponential growth.  
\end{thm}

\noindent
Here, $\nu$ is a weight function $V(\Gamma) \to \mathbb{N}$ and $\Gamma(\nu)$ denotes the graph product $\Gamma \mathcal{Z}$ where $\mathcal{Z}:= \{ \mathbb{Z}_{\nu(u)} \mid u \in V(\Gamma)\}$. (We recall the definition of graph products in the next section.)

\medskip \noindent
In other words, Theorem~\ref{thm:CoarseSepGP} allows us to reduce a separability problem in a right-angled Artin group into a separability problem in a graph product of finite groups. Working with graph products of finite groups may be easier, especially because they can be hyperbolic (while right-angled Artin groups are essentially never hyperbolic). 

\medskip \noindent
Our strategy to prove Theorem~\ref{thm:CoarseSepGP} will be to combine Theorem~\ref{thm:SeparatingAcylHyp} with the following statement:

\begin{thm}\label{thm:MorseSubBuilding}
For every Morse gauge $M$ and every constant $k_0 \geq 0$, there exist constants $A >0, B \geq 0$ such that the following holds. Let $\Gamma$ be a finite graph and $\mathcal{G}$ a collection of finitely generated groups indexed by $V(\Gamma)$. Assume that, for every $u \in V(\Gamma)$:
\begin{itemize}
	\item if $G_u$ is infinite, then so is $\langle \mathrm{link}(u) \rangle$;
	\item $G_u$ has cardinality at least $k_0$.
\end{itemize}
Every bi-infinite $M$-Morse geodesic in $\Gamma \mathcal{G}$ is contained in an $(A,B)$-quasi-isometrically embedded copy of $\Gamma(\nu)$ for some weight $\nu$ with $\min(\nu) \geq k_0$. 
\end{thm}

\noindent
The proof of Theorem~\ref{thm:MorseSubBuilding} will be based on the quasi-median geometry of graph products, which we describe in Section~\ref{section:QM} (following \cite{QM}). The main tool that we will need is the notion of \emph{invasive subgraph}, which we introduced and study in Section~\ref{section:Invasive}. Finally, we prove Theorem~\ref{thm:MorseSubBuilding} in Section~\ref{section:ProofGP}.

\subsection{Quasi-median geometry}\label{section:QM}

\noindent
In this section, we record some basic definitions and properties related to quasi-median graphs and graph products of groups. The main reference for the results here is \cite{QM}.

\paragraph{Quasi-median graphs.} There are many possible equivalent definitions of quasi-median graphs. A short one is the following. First, recall that a \emph{Hamming graph} is a (connected component of a) product of complete graphs. Then, \emph{quasi-median} graphs are retracts of Hamming graphs. This definition motivates the analogy with \emph{median graphs}, also known as one-skeletons of CAT(0) cube complexes, which can be characterised as retracts of hypercubes. However, it may not be the easiest definition to manipulate. In practice, it may be more convenient to define quasi-median graphs as \emph{weakly modular} graphs with no induced copy of $K_{2,3}$ or $K_4^-$, but this may sound more artificial and may be more difficult to digest. This characterisation can be found in \cite{MR1297190}, and this is the perspective taken in \cite{QM}. Since we will not use this definition here, we do not recall it. However, let us mention one piece of the definition of weakly modular graphs for future use:

\begin{lemma}
A quasi-median graph $X$ satisfies the \emph{triangle condition}, i.e.\ for all vertices $o,x,y \in V(X)$ satisfying $d(o,x)=d(o,y)$ and $d(x,y)=1$, there exists a common neighbour $z$ of $x$ and $y$ such that $d(o,z)=d(o,x)-1$. \qed
\end{lemma}

\noindent
Another rather easy consequence of the definition of quasi-median graphs as specific weakly modular graphs is:

\begin{lemma}[{\cite[Lemma~3.3]{Mediangle}}]\label{lem:ByPass}
Let $X$ be a quasi-median graph. If $\gamma_1$ and $\gamma_2$ are two geodesics with the same endpoints, then $\gamma_2$ can be obtained from $\gamma_1$ by bypassing finitely many squares. 
\end{lemma}

\noindent
We refer the reader to \cite{MR1297190} for other characterisations of quasi-median graphs. As for median graphs, it is more informative to introduce \emph{hyperplanes} and to describe how they are related to the geometry of quasi-median graphs.

\begin{definition}
Let $X$ be a quasi-median graph. Two edges $a,b \in E(X)$ are
\begin{itemize}
	\item \emph{parallel} if there exists $e_1=a, \ldots, e_n=b \in E(X)$ such that $e_i$ and $e_{i+1}$ are opposite sides of a square in $X$ for every $1 \leq i \leq n-1$;
	\item \emph{in the same hyperplane} if there exists $e_1=a, \ldots, e_n=b \in E(X)$ such that $e_i$ and $e_{i+1}$ are parallel or belong to the same triangle for every $1 \leq i \leq n-1$.
\end{itemize}
A \emph{hyperplane} is a maximal collection of edges pairwise all in the same hyperplane. We denote by $\mathrm{Hyp}(X)$ the set of the hyperplanes of $X$. 
\end{definition}

\noindent
See Figure~\ref{HypEx} for some examples of hyperplanes. Our next result motivates the idea that the geometry of a quasi-median graph is encoded in the combinatorics of its hyperplanes. 
\begin{figure}
\begin{center}
\includegraphics[trim=0 16.5cm 10cm 0,clip,width=0.6\linewidth]{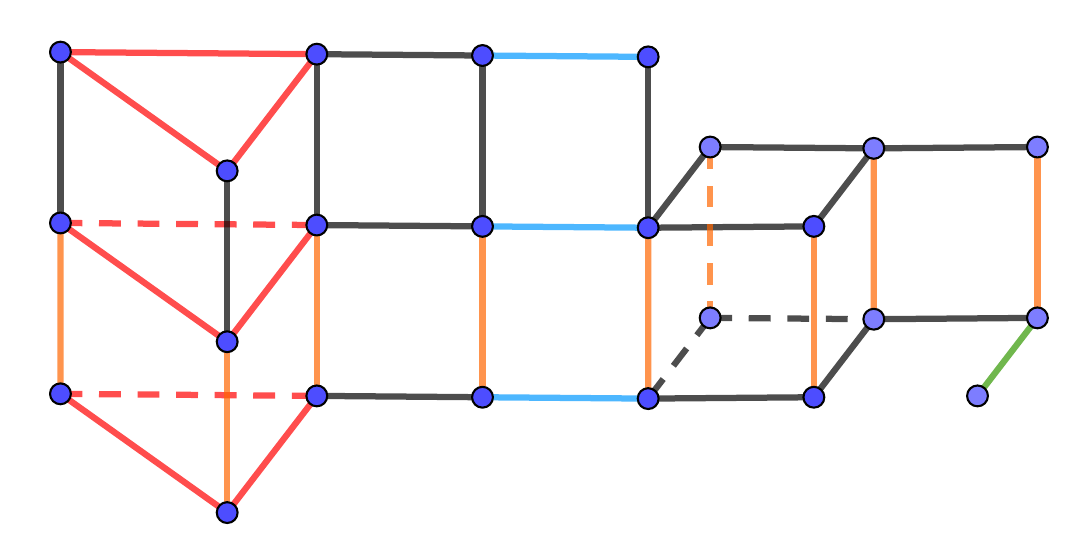}
\caption{A few hyperplanes in a quasi-median graph.}
\label{HypEx}
\end{center}
\end{figure}

\begin{thm}[\cite{QM}]\label{thm:HypQM}
Let $X$ be a quasi-median graph.
\begin{itemize}
	\item Every hyperplane $J$ separates, i.e.\ the graph $X \backslash \backslash J$ obtained from $X$ by removing all the edges of $J$ is disconnected. Its connected components are \emph{sectors}. Sectors are gated in $X$.
	\item For every hyperplane $J$, its \emph{carrier} $N(J)$ (i.e.\ the subgraph in $X$ induced by $J$) and its \emph{fibres} (i.e.\ the connected components of $N(J) \backslash \backslash J$) are gated. 
	\item For every hyperplane $J$, its carrier $N(J)$ decomposes as $C \times F$ for all clique $C$ and fibre $F$ of $J$.
	\item A path is a geodesic if and only if it crosses each hyperplane at most once. Consequently, the distance between two vertices coincides with the number of hyperplanes that separate them. 
\end{itemize}
\end{thm}

\noindent
Recall that, given a graph $X$, a subgraph $Y \leq X$ is \emph{gated} if every vertex $x \in V(X)$ admits a \emph{gate} in $Y$, i.e.\ a vertex $y \in V(Y)$ such that every vertex of $Y$ is connected to $x$ by some geodesic passing through $y$. The map $V(X) \to V(Y)$ given by $x \mapsto (\text{gate in } Y)$ is referred to as the \emph{(gate-)projection onto $Y$}.

\paragraph{Systems of local metrics.} Let $X$ be a graph. Assume that, for each \emph{clique} $C$ of $X$ (i.e.\ every maximal complete subgraph), we fix a metric $\delta_C$ on $V(C)$. We refer to such a data $\{ (C,\delta_C) \mid C \text{ clique}\}$ as a \emph{system of (local) metrics}. Then, we define a \emph{global (pseudo-)metric} $\delta$ on $V(X)$ as
$$\delta (x,y) = \inf \left\{ \sum\limits_{i=1}^{n-1} \delta_{C_i}(x_i,x_{i+1}) \mid \begin{array}{c} x_1=x, \ldots, x_n=y \in V(X) \text{ such that } x_i \\   \text{and } x_{i+1} \text{ belong to } C_i \text{ for every } 1 \leq i \leq n-1 \end{array} \right\}$$
for all $x,y \in V(X)$. 
In full generality, there is no reason for the global metric to extend the local metric or to be compatible in some way with the geometry of the whole graph. In practice, given a quasi-median graph (or, more generally, a \emph{paraclique graph}; see \cite{Contracting}), we ask our system of local metrics to be \emph{coherent}, i.e.\ we want, for any two clique $A$ and $B$ in the same hyperplane, the gate-projection $A \to B$ to induce an isometry $(A,\delta_A) \to (B,\delta_B)$. 

\medskip \noindent
Under this assumption, it is possible to describe the global metric in a way that is more connected with the structure of our quasi-median graph. Given a hyperplane $J$, we can fix  a clique $C$ in $J$ and define the pseudo-metric
$$\delta_J : (x,y) \mapsto \delta_C(\mathrm{proj}_C(x), \mathrm{proj}_C(y))$$
on our quasi-median graph. Because the system of metrics is coherent, $\delta_J$ does not depent on the choice of $C$. Then:

\begin{lemma}[{\cite[Proposition~3.13]{QM}}]\label{lem:DeltaJ}
Let $X$ be a quasi-median graph and $\{(C,\delta_C) \mid C \text{ clique}\}$ a coherent system of local metrics. For all vertices $x,y \in V(X)$,
$$\delta(x,y)= \sum\limits_{J \text{ hyperplane separating $x$ and $y$}} \delta_J(x,y).$$
\end{lemma}

\noindent
It is worth mentioning that, given a (coherent) system of graph-metrics, the global metric is also a graph-metric. Thus, we will often think of $(X,\delta)$ as a graph, as in the next lemmas. 

\begin{lemma}[{\cite[Lemma~5.16]{Quandle}}]\label{lem:SystemGated}
Let $X$ be a quasi-median graph and $\{(C,\delta_C) \mid C \text{ clique}\}$ a coherent system of graph-metrics. For every gated subgraph $Y \leq X$, the subgraph $(Y,\delta)$ is gated in $(X,\delta)$. In particular, $(Y,\delta)$ is convex.
\end{lemma}

\begin{lemma}\label{lem:ProductSyst}
Let $X$ be a quasi-median graph and $\{(C,\delta_C) \mid C \text{ clique}\}$ a coherent system of graph-metrics. If $Y \leq X$ is a gated subgraph that splits a product $Y_1 \times Y_2$ of two subgraphs $Y_1,Y_2 \leq Y$, then $(Y,\delta)= (Y_1, \delta_1) \times (Y_2, \delta_2)$. 
\end{lemma}

\begin{proof}
Let $x,y \in V(Y)$ be two vertices, which we can write as $x=(x_1,x_2)$, $y=(y_1,y_2)$. Let $o$ denote the unique vertex in $Y_1 \cap Y_2$. Thus, $Y_1$ (resp.\ $Y_2$) is identified with the factor $Y_1 \times \{o\}$ (resp.\ $\{o\} \times Y_2$) in $Y_1 \times Y_2$. Consequently, the coordinates $x_1$ and $y_1$ (resp.\ $x_2$ and $y_2$) can be identified with $(x_1,o)$ and $(y_1,o)$ (resp.\ $(o,x_2)$ and $(o,y_2)$). 

\medskip \noindent
A hyperplane $J$ of $Y$ separating $x$ and $y$ crosses either $Y_1$ or $Y_2$. Say that $J$ crosses $Y_1$. Notice that $J$ does not separated $(x_1,x_2)$ and $(x_1,o)$ nor $(y_1,y_2)$ and $(y_1,o)$, hence $\delta_J(x,y) = \delta_J(x_1,y_1)$. Similarly, $\delta_J(x,y)=\delta(x_2,y_2)$ whenever $J$ crosses $Y_2$. We conclude thanks to Lemma~\ref{lem:SystemGated} that
$$\begin{array}{lcl} \delta(x,y) & = & \displaystyle \sum\limits_{J \text{ hyperplane separating $x$ and $y$}} \delta_J(x,y) \\ \\ & = & \displaystyle \sum\limits_{J \text{ hyperplane of } Y_1} \delta_J(x,y) + \sum\limits_{J \text{ hyperplane of } Y_2} \delta_J(x,y) \\ \\ & = & \displaystyle \sum\limits_{J \text{ hyperplane of } Y_1} \delta_J(x_1,y_1) + \sum\limits_{J \text{ hyperplane of } Y_2} \delta_J(x_2,y_2) \\ \\ & = & \delta_1(x,y)+ \delta_2(x,y). \end{array}$$
This proves our lemma. 
\end{proof}

\paragraph{Graph products.} Let $\Gamma$ be a graph and $\mathcal{G}=\{G_u \mid u \in V(\Gamma)\}$ a collection of groups indexed by the vertex-set $V(\Gamma)$ of $\Gamma$. The \emph{graph product} $\Gamma \mathcal{G}$ is the group defined by the relative presentation
$$\langle G_u \ (u \in V(\Gamma)) \mid [G_u,G_v]=1 \ (\{u,v\} \in E(\Gamma)) \rangle$$
where $E(\Gamma)$ denotes the edge-set of $\Gamma$ and where $[G_u,G_v]=1$ is a shorthand for: $[g,h]=1$ for all $g \in G_u$ and $h \in G_v$. 

\medskip \noindent
\textbf{Convention:} In the rest of the article, we will always assume that the factors of our graph products are non-trivial. 

\medskip \noindent
Every element of $\Gamma \mathcal{G}$ can be described as a product $s_1 \cdots s_n$ of elements in $\bigcup_{u \in V(\Gamma)} G_u$. We refer to $s_1 \cdots s_n$ as a \emph{word} and to the $s_i$ as its \emph{syllables}. Clearly, the element of $\Gamma \mathcal{G}$ represented by our word $s_1 \cdots s_n$ remains the same after the application of one of the following moves:
\begin{description}
	\item[(Cancellation)] if there exists $1 \leq i \leq n$ such that $s_i=1$, remove the syllable $s_i$;
	\item[(Merging)] if there exists $1 \leq i \leq n-1$ such that $s_i,s_{i+1} \in G_u$ for some $u \in V(\Gamma)$, replace the subword $s_is_{i+1}$ with the single syllable $(s_is_{i+1})$;
	\item[(Shuffle)] if there exists $1 \leq i \leq n-1$ such that $s_i \in G_u$ and $s_{i+1} \in G_v$ for some $\{u,v\} \in E(\Gamma)$, replace with the subword $s_is_{i+1}$ with the subword $s_{i+1}s_i$. 
\end{description}
A word that cannot be shortened by applying these elementary moves is \emph{graphically reduced}. According to the next statement, every element of $\Gamma \mathcal{G}$ is represented by a graphically reduced word that is unique in a resaonable sense: 

\begin{lemma}[\cite{GreenGP}]\label{lem:GPWP}
Let $\Gamma$ be a graph and $\mathcal{G}$ a collection of groups indexed by $V(\Gamma)$. If $w_1$ and $w_2$ are two graphically reduced words representing the same element of $\Gamma \mathcal{G}$, then $w_2$ can be obtained from $w_1$ by switching successive pairs of syllables that belong to adjacent vertex-groups. 
\end{lemma}

\noindent
The connection between graph products and quasi-median graphs is the following. Given a graph $\Gamma$ and a collection of groups $\mathcal{G}$, the Cayley graph 
$$\mathrm{QM}(\Gamma,\mathcal{G}):= \mathrm{Cayl} \left( \Gamma \mathcal{G}, \bigcup \mathcal{G} \right)$$
is quasi-median \cite[Proposition~8.2]{QM}. Like in any Cayley graph, paths are naturally labelled by words of generators. We can use this perspective in order to recognise geodesics:

\begin{lemma}[{\cite[Lemma~8.3]{QM}}]\label{lem:GeodInQMGP}
Let $\Gamma$ be a graph and $\mathcal{G}$ a collection of groups indexed by $V(\Gamma)$. A path in $\mathrm{QM}(\Gamma, \mathcal{G})$ is a geodesic if and only if it is labelled by a graphically reduced word. 
\end{lemma}

\noindent
Given a graph $\Gamma$ and a collection of groups $\mathcal{G}$ indexed by $V(\Gamma)$, every edge of $\mathrm{QM}(\Gamma, \mathcal{G})$ is labelled by a generator in $\bigcup \mathcal{G}$, and consequently by the vertex of $\Gamma$ that labels the corresponding factor. The next assertion shows that this labelling of the edges of $\mathrm{QM}(\Gamma, \mathcal{G})$ extends to hyperplanes.

\begin{lemma}[{\cite[Lemmas~8.5 and~8.8]{QM}}]
Let $\Gamma$ be a graph and $\mathcal{G}$ a collection of groups indexed by $V(\Gamma)$. If two edges of $\mathrm{QM}(\Gamma , \mathcal{G})$ belong to the same hyperplane, then they are labelled by the same vertex of $\Gamma$. 
\end{lemma}

\noindent
In the sequel, the canonical labelling or colouring of the hyperplanes of $\mathrm{QM}(\Gamma, \mathcal{G})$ by $V(\Gamma)$ will refer to this extension.

\medskip \noindent
Of course, given a finite graph $\Gamma$ and a collection of finitely generated groups $\Gamma \mathcal{G}$, our quasi-median graph $\mathrm{QM}(\Gamma, \mathcal{G})$ is typically not quasi-isometric to the graph product $\Gamma \mathcal{G}$, since this is a Cayley graph constructed from a usually infinite generating set. But it is possible to endow $\mathrm{QM}(\Gamma, \mathcal{G})$ with a system of local metrics such that the global metric recovers $\Gamma \mathcal{G}$. The local metrics we have to take are easy to guess since:

\begin{lemma}[{\cite[Lemma~8.6]{QM}}]\label{lem:GPcliqueQM}
Let $\Gamma$ be a graph and $\mathcal{G}$ a collection of groups. The cliques of $\mathrm{QM}(\Gamma, \mathcal{G})$ are given by the cosets $gG_u$ for $g \in \Gamma\mathcal{G}$ and $u \in V(\Gamma)$. 
\end{lemma}

\noindent
Since cliques are cosets of factors, it is then natural to define metrics on cliques by using finite generating sets of the corresponding factors, namely:

\begin{prop}[{\cite[Claim~8.24]{QM}}]\label{prop:GPlocal}
Let $\Gamma$ be a finite graph and $\mathcal{G}$ a collection of finitely generated groups indexed by $V(\Gamma)$. For every $u \in V(\Gamma)$, fix a finite generating set $S_u$ of $G_u$. For every clique $C=gG_u$ of $\mathrm{QM}(\Gamma,\mathcal{G})$ and all vertices $x,y \in V(C)$, write $x=ga$ and $y=gb$ as graphically reduced words and define $\delta_C(x,y):= d_{S_u}(a,b)$. Then, $\{ (C,\delta_C) \mid C \text{ clique}\}$ is a coherent sytem of metrics and $(\mathrm{QM}(\Gamma, \mathcal{G}),\delta)$ is isometric to the Cayley graph $\mathrm{Cayl} \left(\Gamma, \mathcal{G}, \bigcup_{u \in V(\Gamma)} S_u \right)$. 
\end{prop}

\subsection{Invasive subgraphs}\label{section:Invasive}

\noindent
In this section, we introduce \emph{invasive subgraphs} in quasi-median graphs. They will play a fundamental role in our proof of Theorem~\ref{thm:MorseSubBuilding}. 

\begin{definition}
Let $X$ be a quasi-median graph. A subgraph $Y \leq X$ is \emph{invasive} if it is convex and if, for every clique $C$ of $X$, $|V(Y \cap C)| =0$ or $\geq 2$.
\end{definition}

\medskip \noindent
\begin{minipage}{0.3\linewidth}
\includegraphics[width=0.9\linewidth]{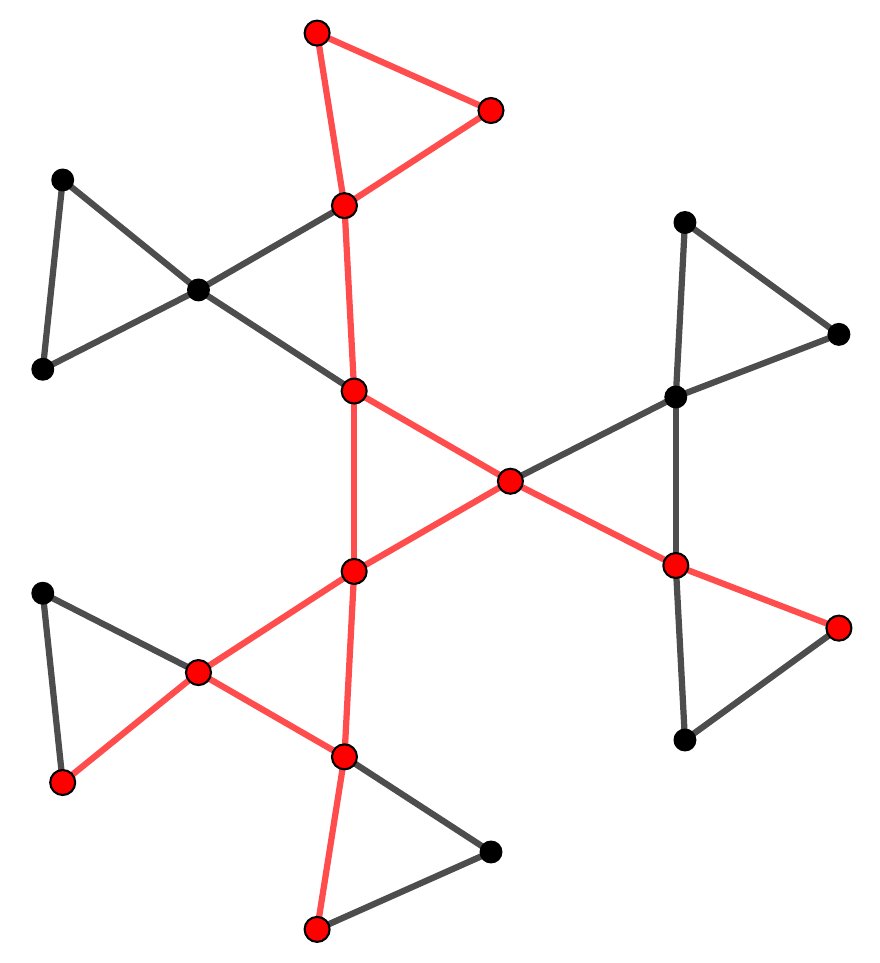}
\end{minipage}
\begin{minipage}{0.68\linewidth}
This definition comes from an analogy between quasi-median graphs and buildings that will be explored by the second-named author in a future work. The idea to keep in mind is that, given a quasi-median graph $X$, maximal convex bipartite subgraphs correspond to appartments and cliques correspond panels. Then, our invasive subgraphs correspond to subbuildings. The figure on the left illustrates an invasive subgraph in a block graph.  
\end{minipage}

\medskip \noindent
We will be especially interested in invasive subgraphs that are regular in the following sense:

\begin{definition}
Let $X$ be a quasi-median graph. Given a subgraph $Y \leq X$ and a hyperplane $J$, the \emph{thickness of $J$ in $Y$} is the number of sectors delimited by $J$ that intersect $Y$. Given a colouring $\kappa : \mathrm{Hyp}(X) \to S$, $Y$ is \emph{$\kappa$-regular} if two hyperplanes crossing $Y$ always have the same thickness in $Y$ whenever they have the same $\kappa$-colour. 
\end{definition}

\noindent
The rest of the section is dedicated to the proof of the following description of invasive subgraphs in quasi-median graphs of graph products:

\begin{prop}\label{prop:InvasiveSubIso}
Let $(\Gamma,\nu)$ be a weighted graph. Endow $\mathrm{QM}(\Gamma,\nu)$ with its canonical colouring $\kappa : \mathrm{Hyp}(\mathrm{QM}(\Gamma,\nu)) \to V(\Gamma)$. Every $\kappa$-regular invasive subgraph $Y$ of $\mathrm{QM}(\Gamma, \nu)$ is isomorphic to  
$$\mathrm{QM}\left( \Gamma,  u \mapsto \text{thickness of the hyperplanes crossing $Y$ coloured } u \right).$$
\end{prop}

\noindent
Here, a \emph{weighted graph} $(\Gamma, \nu)$ refers to a graph $\Gamma$ endowed with a \emph{weight} $\nu$ that associates to every vertex of $\Gamma$ a cardinal. Then, $\mathrm{QM}(\Gamma,\nu)$ refers to the quasi-median graph of the graph product $\Gamma \mathcal{G}$ of $\mathcal{G}= \{ G_u \mid u \in V(\Gamma)\}$ where each $G_u$ has cardinality $\nu(u)$ (say a cyclic group, but the choice actually does not matter since the quasi-median graph, up to graph-isomrophism, is not sensible to the algebraic structures of the factors, only to their cardinalities). 

\begin{proof}[Proof of Proposition~\ref{prop:InvasiveSubIso}.]
For every vertex $u \in V(\Gamma)$, we denote by $\tau(u)$ the thickness of the hyperplanes crossing $Y$ coloured $u$. We fix an arbitrary group $G_u$ (resp.\ $H_u$) of cardinality $\tau(u)$ (resp.\ $\nu(u)$) for every $u \in V(\Gamma)$, and we set $\mathcal{G}:= \{G_u \mid u \in V(\Gamma)\}$ (resp.\ $\mathcal{H}:= \{H_u \mid u \in V(\Gamma) \}$). For convenience, we identify $\mathrm{QM}(\Gamma, \nu)$ with $\mathrm{QM}(\Gamma, \mathcal{G})$. Our goal is to prove that $Y$ is isomorphic to $\mathrm{QM}(\Gamma, \mathcal{H})$.

\medskip \noindent
We fix a basepoint $o \in V(Y)$; and, for every hyperplane $J$ crossing $Y$, we fix a bijection $\sigma_J$ from the sectors delimited by $J$ intersecting $Y$ to $H_{\kappa(J)}$, which we can choose so that $o \in \sigma_J^{-1}(1)$. For every geodesic $\gamma$ with vertices $x_0=o, x_1, \ldots, x_n \in V(Y)$, we define its label as
$$\varphi(\gamma):= \prod\limits_{i=0}^{n-1} \sigma_{\text{hyperplane containing } \{x_i,x_{i+1}\}} ( \text{sector containing } x_{i+1}),$$
thought of as a word written over $\bigcup \mathcal{H}$. In other words, thinking of $\sigma_J(S)$ as a cost to cross a hyperplane $J$ from the sector containing $o$ to the sector $S$, $\varphi(\gamma)$ is the product of all the costs given by the hyperplanes that $\gamma$ successively crosses. A key observation is that:

\begin{claim}\label{claim:GeodSameLabel}
If $\gamma_1,\gamma_2$ are two geodesics connecting $o$ to the same vertex of $Y$, then $\varphi(\gamma_1)= \varphi(\gamma_2)$ in $\Gamma \mathcal{H}$.
\end{claim}

\noindent
According to Lemma~\ref{lem:ByPass}, it suffices to assume that $\gamma_2$ can be obtained from $\gamma_1$ by bypassing a square. In other words, 
$$\gamma_1 : x_0=o, \ x_1, \ldots, x_{i-1}, \ x_i, \ x_{i+1}, \ldots, x_n$$
where $x_{i-1},x_i,x_{i+1}$ span a square; and, if $x_i'$ denotes the fourth vertex of this square, then
$$\gamma_2 : x_0=o, \ x_1, \ldots, x_{i-1}, \ x_i', \ x_{i+1}, \ldots, x_n.$$
For every $0 \leq k \leq n-1$, let $J_k$ denote the hyperplane containing the edge $\{x_k,x_{k+1}\}$. Since the hyperplane $J_{i-1}$ (resp.\ $J_i$) contains both $\{x_{i-1},x_i\}$ and $\{x_i',x_{i+1}\}$ (resp.\ $\{x_i,x_{i+1}\}$ and $\{x_{i-1},x_i'\}$), we have 
$$\varphi(\gamma_1)= \sigma_{J_0}(x_1) \cdots \sigma_{J_{i-2}}(x_{i-1}) \sigma_{J_{i-1}}(x_i) \sigma_{J_i}(x_{i+1}) \sigma_{J_{i+1}}(x_{i+2}) \cdots \sigma_{J_{n-1}}(x_n)$$
and 
$$\varphi(\gamma_2)= \sigma_{J_0}(x_1) \cdots \sigma_{J_{i-2}}(x_{i-1}) \sigma_{J_{i}}(x_i') \sigma_{J_{i-1}}(x_{i+1}) \sigma_{J_{i+1}}(x_{i+2}) \cdots \sigma_{J_{n-1}}(x_n),$$
where $\sigma_{J_{i-1}}(x_i)= \sigma_{J_{i-1}}(x_{i+1})$ (resp.\ $\sigma_{J_i}(x_i') = \sigma_{J_i}(x_{i+1})$) because $x_i$ and $x_{i+1}$ (resp.\ $x_i'$ and $x_{i+1}$) belong to the same sector delimited by $J_{i-1}$ (resp.\ $J_i$). In other words, $\varphi(\gamma_2)$ is obtained from $\varphi(\gamma_1)$ by switching the two syllables $\sigma_{J_{i-1}}(x_i)$ and $\sigma_{J_i}(x_{i+1})$. But these two syllables commute in $\Gamma \mathcal{H}$. Indeed, since the edges $\{x_i,x_{i-1}\}$ and $\{x_i,x_{i+1}\}$ span a square in $\mathrm{QM}(\Gamma, \mathcal{G})$, our two edges are labelled by generators that belong to two adjacent vertex-groups in $\mathcal{G}$, which implies that $\sigma_{J_{i-1}}(x_i)$ and $\sigma_{J_i}(x_{i+1})$ belong to two adjacent vertex-groups in $\mathcal{H}$. This concludes the proof of Claim~\ref{claim:GeodSameLabel}. 

\medskip \noindent
Claim~\ref{claim:GeodSameLabel} allows us to define
$$\Phi: \left\{ \begin{array}{ccc} V(Y) & \to & V(\mathrm{QM}(\Gamma, \mathcal{H})) \\ y & \mapsto & \eta(\text{geodesic from $o$ to $y$}) \end{array} \right..$$
Our goal now is to verify that $\Xi$ induces a graph-isomorphism. We start by proving two preliminary observations.

\begin{claim}\label{claim:GeodGraphRed}
For every geodesic $\gamma$ connecting $o$ to some vertex of $Y$, $\varphi(\gamma)$ is graphically reduced.
\end{claim}

\noindent
Let $x_0=o, \ldots, x_n$ denote the successive vertices of $\gamma$. For every $0 \leq i \leq n-1$, let $J_i$ denote the hyperplane containing $\{x_i,x_{i+1}\}$. If $\varphi(\gamma)= \sigma_{J_0}(x_1) \cdots \sigma_{J_{n-1}}(x_n)$ is not graphically reduced, then we can find two indices $0 \leq i< j \leq n-1$ such that $\sigma_{J_i}(x_{i+1})$ and $\sigma_{J_j}(x_{j+1})$ belong to the same vertex-group (in $\mathcal{H}$), which is adjacent to all the vertex-groups containing $\sigma_{J_k}(x_{k+1})$ for $i<k<j$. This implies that the edges $\{x_i,x_{i+1}\}$ and $\{x_j,x_{j+1}\}$ are labelled by generators that belong to the same vertex-group (in $\mathcal{G}$), which is adjacent to all the vertex-groups containing the generators that label $\{x_k,x_{k+1}\}$ for $i<k<j$. In other words, our geodesic $\gamma$ is labelled by a word that is not graphically reduced, contradicting the fact that $\gamma$ is a geodesic according to Lemma~\ref{lem:GeodInQMGP}. This concludes the proof of Claim~\ref{claim:GeodGraphRed}. 

\begin{claim}\label{claim:WordSurj}
For every $y \in Y$ and for every graphically reduced word $w$ representing $\Phi(y)$ in $\Gamma \mathcal{H}$, there exists a geodesic $\gamma$ connecting $o$ to $y$ such that $\varphi(\gamma)=w$. 
\end{claim}

\noindent
Fix a geodesic $\eta$ connecting $o$ to $y$. If $x_0=o, \ldots, x_n$ denote the successive vertices of $\eta$ and if $J_i$ denotes the hyperplane containing $\{x_i,x_{i+1}\}$ for every $0 \leq i \leq n-1$, then we know from Claim~\ref{claim:GeodGraphRed} that $\varphi(\eta) = \sigma_{J_0}(x_1) \cdots \sigma_{J_{n-1}}(x_n)$ is graphically reduced. As a consequence of Lemma~\ref{lem:GPWP}, it suffices to assume that $w$ can be obtained from $\varphi(\eta)$ by switching two consecutive syllables that belong to adjacent vertex-groups, say
$$w= \sigma_{J_0}(x_1) \cdots \sigma_{J_{i-1}}(x_i) \sigma_{J_{i+1}}(x_{i+2}) \sigma_{J_i}(x_{i+1}) \sigma_{J_{i+2}}(x_{i+3}) \cdots \sigma_{J_{n-1}}(x_n)$$
for some $0 \leq i \leq n-2$. Since $\sigma_{J_i}(x_{i+1})$ and $\sigma_{J_{i+1}}(x_{i+2})$ belong to adjacent vertex-groups (in $\mathcal{H}$), necessarily the edges $\{x_i,x_{i+1}\}$ and $\{x_{i+1},x_{i+2}\}$ are labelled by generators that belong to adjacent vertex-groups (in $\mathcal{G}$). As a consequence, $\{x_{i+1},x_i\}$ and $\{x_{i+1},x_{i+2}\}$ span a square. By passing this square yields a new geodesic $\gamma$ from $\eta$. By construction, $\varphi(\gamma)=w$, concluding the proof of Claim~\ref{claim:WordSurj}. 

\medskip \noindent
Now, we are ready to verify that $\Phi$ is injective. So let $x,y \in V(Y)$ be two vertices such that $\Phi(x)= \Phi(y)$. Fix a graphically reduced word $w$ representing this common element. According to Claim~\ref{claim:WordSurj}, we can find two geodesic $\alpha$ and $\beta$ connecting $o$ to $x$ and $y$ respectively such that $\varphi(\alpha)=w=\varphi(\beta)$. If $\alpha \neq \beta$, then we can write $\alpha= \zeta \cup \alpha'$ and $\beta = \zeta \cup \beta'$ where $\zeta$ is a common subpath of $\alpha$ and $\beta$, and where the first edges of $\alpha'$ and $\beta'$ differ. In $\varphi(\alpha)$ and $\varphi(\beta)$, we find a common prefix, namely $\varphi(\zeta)$, followed by a syllable that differ in $\varphi(\alpha)$ and $\varphi(\beta)$. Thus, we must have $\alpha = \beta$, which implies that $x=y$, as desired. 

\medskip \noindent
Next, assume that $x,y \in V(Y)$ are two adjacent vertices. If $d(o,x)<d(o,y)$, then there exists a geodesic $\gamma$ connecting $o$ to $y$ and passing through $x$. Then, because $\varphi(\gamma)$ represents $\Phi(y)$ and $\varphi(\gamma \backslash \{\text{last edge}\})$ represents $\Phi(x)$ in $\Gamma \mathcal{H}$, it is clear that $\Phi(x)$ and $\Phi(y)$ are adjacent in $\mathrm{QM}(\Gamma, \mathcal{H})$. If $d(o,x)>d(o,y)$, the argument is symmetric. If $d(o,x)=d(o,y)$, then we know from the triangle condition that there exists a geodesic $\gamma$ connecting $o$ to a common neighbour $z$ of $x$ and $y$ such that $\gamma \cup \{z,x\}$ and $\gamma \cup \{z,y\}$ are geodesics. Let $J$ denote the hyperplane containing the triangle defined by $x,y,z$. We have
$$\varphi(\gamma \cup \{z,x\}) = \varphi(\gamma) \sigma_J(x) \text{ and } \varphi(\gamma \cup \{z,y\}) = \varphi(\gamma) \sigma_J(y).$$
Since $\sigma_J(x)$ and $\sigma_J(y)$ belong to the same vertex-group in $\mathcal{H}$, we conclude that $\Phi(x)$ and $\Phi(y)$ are adjacent in $\mathrm{QM}(\Gamma, \mathcal{H})$.

\medskip \noindent
Conversely, assume that $x,y \in V(Y)$ are two vertices with $\Phi(x)$ and $\Phi(y)$ adjacent in $\mathrm{QM}(\Gamma, \mathcal{H})$. If $d(1, \Phi(x))< d(1, \Phi(y))$, let $w$ denote the graphically reduced word labelling a geodesic that connects $1$ to $\Phi(y)$ passing through $\Phi(x)$. According to Claim~\ref{claim:WordSurj}, there exists a geodesic $\gamma$ connecting $o$ to $y$ such that $\varphi(\gamma) = w$. Since $\varphi(\gamma \backslash \{\text{last edge}\}) = w \backslash \{\text{last syllable}\}$ represents $\Phi(x)$, necessarily $x$ is the penultimate vertex of $\gamma$, proving that $x$ and $y$ are adjacent. If $d(1, \Phi(x))>d(1, \Phi(y))$, the argument is symmetric. If $d(1, \Phi(x))= d(1, \Phi(y))$, then we deduce from the triangle condition satisfied by $\mathrm{QM}(\Gamma, \mathcal{H})$ and from Lemma~\ref{lem:GeodInQMGP} that there exist two graphically reduced words $wa$ and $wb$ representing $\Phi(x)$ and $\Phi(y)$ respectively, where $a$ and $b$ are two syllables that belong to the same vertex-group. According to Claim~\ref{claim:WordSurj}, there exists a geodesic $\gamma$ connecting $o$ to some vertex $z \in V(Y)$ such that $\varphi(\gamma)=w$. Arguing as before, it is clear that $z$ is common neighbour of $x$ and $y$. Moreover, because $a$ and $b$ belong to the same vertex-group in $\mathcal{H}$, the edges $\{z,x\}$ and $\{z,y\}$ must be labelled by generators that belong to the same vertex-group in $\mathcal{G}$. We conclude that $x$ and $y$ are adjacent in $\mathrm{QM}(\Gamma, \mathcal{G})$, as desired.

\medskip \noindent
Thus, we have proved that $\Phi$ induces a graph-embedding $Y \to \mathrm{QM}(\Gamma, \mathcal{H})$. So far, we have only used the fact that $Y$ is a $\kappa$-regular convex subgraph. In order to prove that $\Phi$ is surjective, which will conclude the proof of our proposition, we need to know that $Y$ is invasive. 

\medskip \noindent
Let $p \in V(\mathrm{QM}(\Gamma, \mathcal{H}))$ be a vertex. We want to find a vertex $y \in V(Y)$ such that $\Phi(y)=p$. We argue by induction over the distance from $1$ to $p$. If $d(1,p)=0$, then $\Phi(o)=1=p$. So assume that $d(1,p) \geq 1$. Thinking of $p$ as a graphically reduced word, we can decompose $p$ as $p_0q$, where $q$ denotes the last letter of $p$ and $p_0$ the corresponding prefix of $p$. We know by induction that there exists $y_0 \in V(Y)$ such that $\Phi(y_0)=p_0$. If $u \in V(\Gamma)$ is the vertex such that $q \in H_u$, let $C$ denote the clique $y_0G_u$. Notice that the hyerplane $J$ containing $C$ necessarily crosses $Y$. Indeed, since $Y$ is invasive and $y_0 \in V(Y \cap C)$, we know that $Y$ must contain at least two vertices of $C$. The vertex we are interested in is the unique vertex $y$ of $C$ that belongs to the sector $S$ satisfying $\lambda_J(S)=q$. If $y$ belongs to $V(Y)$, then it is clear that $\Phi(y)=p_0q = p$, just by construction.

\medskip \noindent
In order to justify that $y \in V(Y)$, we start by fixing a vertex $z \in V(Y \cap S)$ and we notice that there exists a geodesic connecting $y_0$ to $z$ that passes through $y$. This follows from the gatedness of sectors. Thus, we conclude from the convexity of $Y$ that $y \in V(Y)$, as desired. 
\end{proof}

\subsection{Proof of Theorem~\ref{thm:CoarseSepGP}}\label{section:ProofGP}

\noindent
Now, we turn to the proof of Theorem~\ref{thm:CoarseSepGP}. In view of Proposition~\ref{prop:InvasiveSubIso}, our goal will be to prove that Morse geodesics are contained in neighbourhoods of invasive subgraphs. We start by stating and proving a sufficient criterion that allows us to deduce that some subgraphs are contained in invasive subgraphs. 

\begin{prop}\label{prop:FindingInvasive}
Let $X$ be a quasi-median graph, $\{(C,\delta_C) \mid C \text{ clique}\}$ a coherent system of graph-metrics, and $V \subset V(X)$ a set of vertices. Given a constant $D \geq 0$, a colouring $\kappa : \mathrm{Hyp}(X) \to S$, and a weight $\nu : S \to \mathbb{N}$, assume that
\begin{itemize}
	\item for every hyperplane $J$, $\mathrm{diam}(V,\delta_J) \leq D$;
	\item $\nu$ is bounded, and every hyperplane $J$ delimits $\geq \nu(\kappa(J))$ sectors.
\end{itemize}
There exists a $\kappa$-regular invasive subgraph $Z \leq X$ containing $V$ such that
\begin{itemize}
	\item the inclusion $Z \hookrightarrow X$ induces a $2\max(D,|\nu|)$-biLipschitz embedding $Z \hookrightarrow (X,\delta)$;
	\item for every hyperplane $J$ crossing $Z$, its thickness in $Z$ is $\geq \nu(\kappa(J))$.
\end{itemize}
\end{prop}

\noindent
We start by showing how to construct invasive subgraphs in full generality: 

\begin{lemma}\label{lem:ConstructingInvasive}
Let $X$ be a quasi-median graph. For every hyperplane $J$, fix a collection $\mathscr{S}_J$ of $\geq 2$ sectors delimited by $J$. The subgraph $Z$ of $X$ induced by
$$\bigcap\limits_{J \text{ hyperplane}} \bigcup\limits_{S \in \mathscr{S}_J} V(S)$$
is an invasive subgraph. Moreover, for every hyperplane $J$ crossing $Z$, the thickness of $J$ in $Z$ coincides with $\# \mathscr{S}_J$. 
\end{lemma}

\begin{proof}
According to \cite[Lemma~2.105]{QM}, each subgraph $Z_J$ induced by $\bigcup_{S \in \mathscr{S}_J} V(S)$ is convex. Therefore, the intersection $Z= \bigcap_J Z_J$ is convex. Next, let $C$ be a clique of $X$ with at least one vertex in $V(C \cap Z)$, and let $J$ denote the hyperplane containing $C$. We claim that, given a sector $S$ delimited by $J$, the following assertions are equivalent:
\begin{itemize}
	\item[(i)] $S$ belongs to $\mathscr{S}_J$;
	\item[(ii)] $S$ intersects $Z$;
	\item[(iii)] the vertex of $C \cap S$ belongs to $Z$.
\end{itemize}
Notice that, as a consequence, we have that $V(C \cap S)| = |\mathscr{S}_J| \geq 2$ and we deduce that thickness of $J$ in $Z$ equals $|\mathscr{S}_J|$. Thus, it suffices to prove our equivalences in order to conclude the proof of our lemma.

\medskip \noindent
The implications $(iii) \Rightarrow (ii) \Rightarrow(i)$ are clear. Now, assume that $(i)$ holds. By assumption, we know that there exists some vertex $x \in V(C \cap Z)$. Let $Y$ denote the vertex of $C \cap S$. Since $J$ is the only hyperplane separating $x$ and $y$, and since $x \in V(Z)$, we know that 
$$y \in \bigcap\limits_{H \text{ hyperplane }\neq J} \bigcup\limits_{R \in \mathscr{S}_H} V(R).$$
But, if $S$ belongs to $\mathscr{S}_J$, we also know that $y \in \bigcup_{R \in \mathscr{S}_J} V(R)$; hence $y \in V(Z)$. Thus, the implication $(i) \Rightarrow (iii)$ is proved. 
\end{proof}

\begin{proof}[Proof of Proposition~\ref{prop:FindingInvasive}.]
Given a hyperplane $J$, we want to define some collection $\mathscr{S}_J$ of sectors delimited by $J$. Notice that, given two pairs of vertices $(a,b)$ and $(a',b')$ such that $a,a'$ belong to the same sector delimited by $J$ and such that $b,b'$ also belong to the same sector delimited by $J$, we have $\delta_J(a,b)=\delta_J(a',b')$. Consequently, it makes sense to talk about the $\delta_J$-distance between two sectors delimited by $J$. Now, fix an arbitrary sector $S_J$ delimited by $J$ that intersects $V$, and define $\mathscr{S}_J$ as the collection of all the sectors at $\delta_J$-distance $\leq \max (D, \nu (\kappa(J)))$ from $S_J$. Notice that, by construction, $\# \mathscr{S}_J$ depends only on the $\kappa$-colour of $J$ and is bounded below by $\sigma(\kappa(J))$. 

\medskip \noindent
We deduce from Lemma~\ref{lem:ConstructingInvasive} that the subgroup $Z$ induced by
$$\bigcap\limits_{J \text{ hyperplane}} \bigcup\limits_{S \in \mathscr{S}_J} V(S)$$
is invasive, $\kappa$-regular; and that, for every hyperplane $J$ crossing $Z$, its thickness in $Z$ is $\geq \nu(\kappa(J))$. Next, given two vertices $x,y \in V(Z)$, we know that
$$d(x,y) \leq \delta(x,y) = \sum\limits_{J \text{ separating $x$ and $y$}} \delta_J(x,y),$$
where the equality is justified by Lemma~\ref{lem:DeltaJ} and where the inequality follows from the assumption that each $\delta_C$ (and consequently each $\delta_J$) is a graph-metric. But, by construction of $Z$, we know that, for every hyperplane $J$, our two vertices $x$ and $y$ belongs to two sectors in $\mathscr{S}_J$, necessarily at distance $\leq 2 \max(D, \nu(\kappa(J))) \leq 2 \max(D,|\nu|)$. Hence
$$\delta(x,y) = 2 \max(D,\nu(\kappa(J)))  \sum\limits_{J \text{ separating $x$ and $y$}} 1 = 2 \max(D, |\nu|) d(x,y), $$
showing that the inclusion map $Z \hookrightarrow X$ indeed induces a $2\max(D,|\nu|)$-biLipschitz embedding $Z \hookrightarrow (X,\delta)$.
\end{proof}

\noindent
The next step is to verify that the criterion provided by Proposition~\ref{prop:FindingInvasive} applies to Morse geodesics. This will be a consequence of the next proposition:

\begin{prop}\label{prop:MorseInterProd}
Given a Morse gauge $M$, there exists a constant $C \geq 0$ such that the following holds. Let $X$ be a graph, $\gamma$ an $M$-Morse geodesic, and $Y \leq X$ an isometrically embedded subgraph. If $Y$ splits as a product of two graphs with infinite diameters, then $\mathrm{diam}(\gamma \cap Y) \leq C$.
\end{prop}

\noindent
The following observation will be needed in our proof of the proposition:

\begin{lemma}\label{lem:MorseHausdorff}
Let $X$ be a graph, $\gamma$ an $M$-Morse geodesic, $\alpha \subset \gamma$ a subsegment, and $\beta$ an $(A,B)$-quasigeodesic with the same endpoints as $\alpha$. The Hausdorff distance between $\alpha$ and $\beta$ is at most
$$1+6A+3B+9AM(A,B).$$
\end{lemma}

\begin{proof}
First, we want to prove that $\beta$ is contained in the $(2A+B+3AM(A,B))$-neighbourhood of $\alpha$. Let $x,y \in \beta$ be two vertices that delimit a maximal subsegment $\beta_0$ of $\beta$ all of whose vertices are at distance $>M(A,B)$ from $\alpha$. If no such subsegment exists, then $\beta$ is contained in the $M(A,B)$-neighbourhood of $\alpha$ and we are done. So we assume that such $x,y \in \beta$ exist. Every vertex of $\beta_0$ lies at distance $>M(A,B)$ from $\alpha$, but it lies at distance $\leq M(A,B)$ from $\gamma$, so at distance $\leq M(A,B)$ from a vertex that belongs to one of the two connected components $\gamma_-,\gamma_+$ of $\gamma \backslash \alpha$. 
\begin{center}
\includegraphics[width=0.5\linewidth]{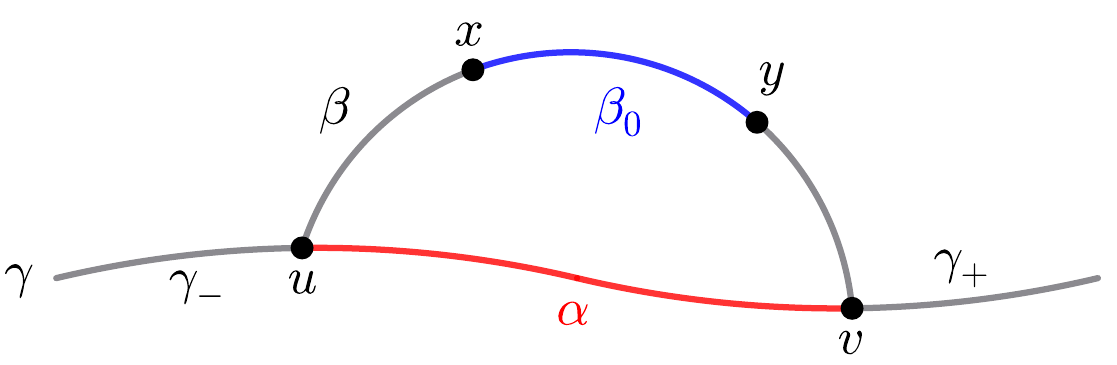}
\end{center}

 \noindent
Assume that there exist $a,b \in \beta_0$ such that $a$ (resp.\ $b$) lies at distance $\leq M(A,B)$ from some vertex $p \in \gamma_-$ (resp.\ $q \in \gamma_+$). Without loss of generality, we can assume that $a$ and $b$ are adjacent. We have
$$\mathrm{length}(\alpha) \leq d(p,q) \leq d(p,a) + d(a,b)+d(b,q) \leq 1+ 2M(A,B),$$
hence
$$\mathrm{length}(\beta) \leq A \cdot \mathrm{length}(\alpha)+B \leq A(1+2M(A,B))+B.$$
Therefore, $\beta$ is contained in the $(A+B+2AM(A,B))$-neighbourhood of $\alpha$.

\medskip \noindent
Otherwise, assume that all the vertices of $\beta_0$ lie at distances $\leq M(A,B)$ from $\gamma_+$. (The case where they all lie at distances $\leq M(A,B)$ from $\gamma_-$ is symmetric.) By maximality of $\beta_0$, necessarily $x$ has a neighbour in $\beta$ that lies at distance $\leq M(A,B)$ from some vertex $p \in \alpha$. Of course, $d(x,p) \leq 1+M(A,B)$. We also know by assumption that $x$ lies at distance $\leq M(A,B)$ from some vertex $q \in \gamma_+$. We have
$$d(p,v) \leq d(p,q) \leq d(p,x) + d(x,q) \leq 1+2M(A,B),$$
hence
$$d(x,v) \leq d(x,p) + d(p,v) \leq 2+3M(A,B).$$
Notice that, given a vertex $w$ of $\beta_0$, the subsegment of $\beta$ delimited by $w$ and $v$, which is contained in the subsegment of $\beta$ delimited by $x$ and $v$, has length at most $Ad(x,v)+B$. Therefore, 
$$d(w,v) \leq A(2+3M(A,B))+B.$$
Thus, we have proved that $\beta$ is contained in the $(2A+B+3AM(A,B))$-neighbourhood of $\alpha$, as desired.

\medskip \noindent
Conversely, we need to show that $\alpha$ is contained in a controlled neighbourhood of $\beta$. Let $z$ be a vertex of $\alpha$. For every vertex $p$ of $\beta$, fix a vertex $\pi(p)$ of $\alpha$ at distance $\leq 2A+B+3AM(A,B)$. Defining an orientation on $\alpha$ such that $\gamma_-$ (resp.\ $\gamma_+$) lies at the left (resp.\ right) of $\alpha$, we fix two successive vertices $a$ and $b$ along $\beta$ such that $\pi(a)$ lies on the left of $z$ and $\pi(b)$ on the right of $z$. We have
$$d(\pi(a),\pi(b)) \leq d(\pi(a),a)+ d(a,b)+ d(b,\pi(b)) \leq 1+ 2(2A+B+3AM(A,B)),$$
hence
$$\begin{array}{lcl} d(z,a) & \leq & d(z,\pi(a))+ d(\pi(a),a) \leq d(\pi(a),\pi(b)) + d(\pi(a),a) \\ \\ & \leq & 1+3 (2A+B+3AM(A,B)). \end{array}$$
Thus, we have proved that $\alpha$ is contained in the $(1+6A+3B+9AM(A,B))$-neighbourhood of $\beta$, as desired. 
\end{proof}

\begin{proof}[Proof of Proposition~\ref{prop:MorseInterProd}.]
For all $A,B \geq 0$, we set $C(A,B):= 1+6A+3B+9AM(A,B)$. Decompose $Y$ as $Y_1 \times Y_2$ where $Y_1$ and $Y_2$ have infinite diameters, and let $p_1:=(a_1,b_1), p_2:= (a_2,b_2)$ be two vertices of $\gamma \cap Y$. 

\medskip \noindent
\begin{minipage}{0.38\linewidth}
\includegraphics[width=0.9\linewidth]{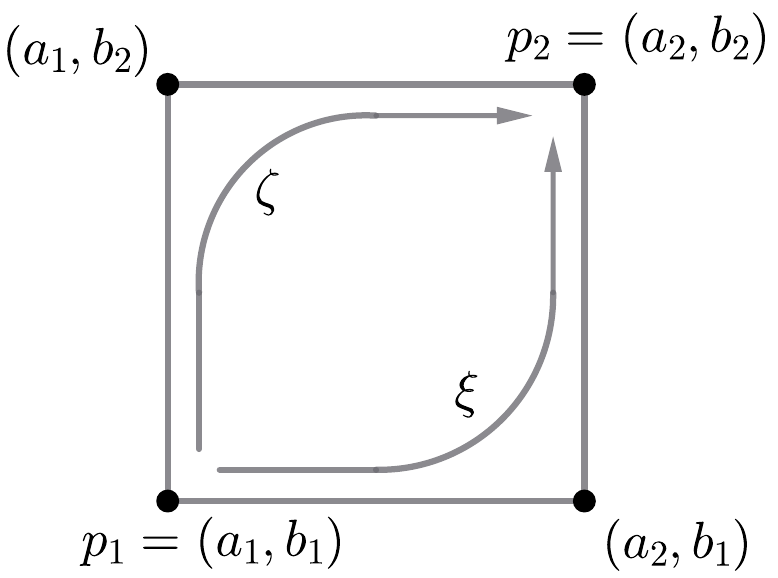}
\end{minipage}
\begin{minipage}{0.6\linewidth}
Let $\zeta$ be a geodesic in $Y$ connecting $p_1$ to $p_2$ as follows. First, we connect $p_1=(a_1,b_1)$ to $(a_1,b_2)$ by a geodesic in $\{a_1\} \times Y_2$; and, then, we connect $(a_1,b_2)$ to $p_2=(a_2,b_2)$ by a geodesic in $Y_1 \times \{b_2\}$. Similarly, we define a geodesic $\xi$ in $Y$ connecting $p_1$ to $p_2$ by first connecting $p_1=(a_1,b_1)$ to $(a_2,b_1)$ through a geodesic in $Y_1 \times \{b_1\}$ and then by connecting $(a_2,b_1)$ to $p_2=(a_2,b_2)$ through a geodesic in $\{a_2\} \times Y_2$. 
\end{minipage}

\medskip \noindent
On the one hand, it is clear that the Hausdorff distance between $\zeta$ and $\xi$ is at least $\min(d(a_1,a_2),d(b_1,b_2))$. On the other hand, it follows from Lemma~\ref{lem:MorseHausdorff} that it must be $\leq 2C(1,0)$. Therefore, either $d(a_1,a_2) \leq 2C(1,0)$ or $d(b_1,b_2) \leq 2C(1,0)$. The two cases being symmetric, we assume that $d(b_1,b_2) \leq 2C(1,0)$. 

\medskip \noindent
\begin{minipage}{0.38\linewidth}
\includegraphics[width=0.9\linewidth]{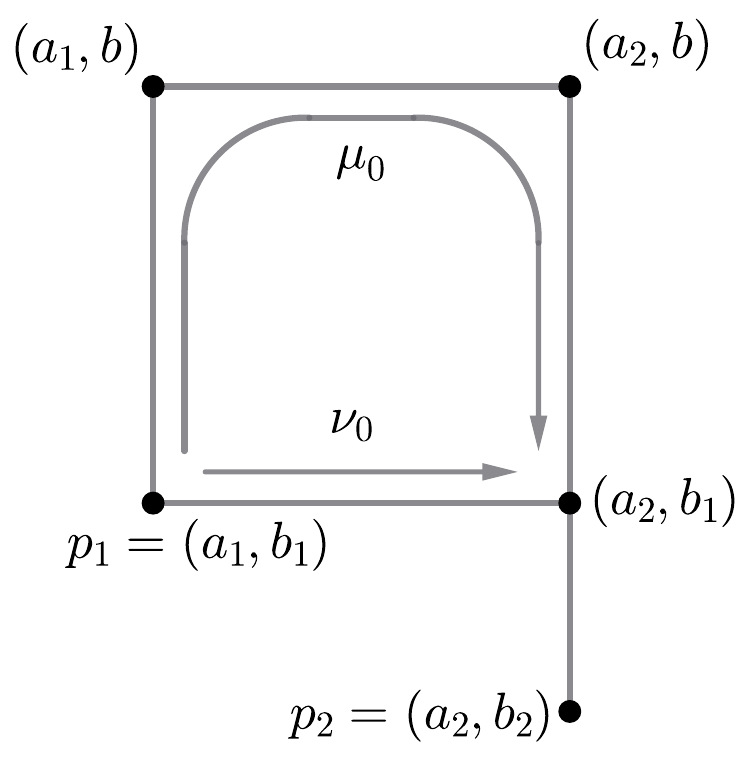}
\end{minipage}
\begin{minipage}{0.6\linewidth}
Let $\nu_0$ be a geodesic connecting $p_1=(a_1,b_1)$ to $(a_2,b_1)$ in $Y_1 \times \{b_1\}$. Because $Y_2$ has infinite diameter, it contains a vertex $b$ at distance $d_{Y_1}(a_1,a_2)$ from $b_1$. Let $\mu_0$ a $(3,0)$-quasigeodesic connecting $p_1=(a_1,b_1)$ to $(a_2,b_1)$ obtained by first connecting $p_1=(a_1,b_1)$ to $(a_1,b)$ through a geodesic in $\{a_1\} \times Y_2$, then by connecting $(a_1,b)$ to $(a_2,b)$ through a geodesic in $Y_1 \times \{b\}$, and finally by connecting $(a_2,b)$ to $(a_2,b_1)$ through a geodesic in $\{a_2\} \times Y_2$. Let $\mu$ and $\nu$ denote the path obtained from $\mu_0$ and $\nu_0$ respectively by concatenating with a geodesic between $(a_2,b_1)$ and $p_2=(a_2,b_2)$. Because $d(b_1,b_2) \leq 2C(1,0)$, $\nu$ is a $(1,2C(1,0))$-quasigeodesic and $\mu$ is a $(3,2C(1,0))$-quasigeodesic.
\end{minipage}

\medskip \noindent
On the one hand, the Hausdorff distance between $\mu$ and $\nu$ is bounded below by $d(b,b_1)=d(a_1,a_2)$. On the other hand, it follows from Lemma~\ref{lem:MorseHausdorff} that it must be at most $C(1,2C(1,0)) + C(3,2C(1,0))$. Hence $d(a_1,a_2) \leq  C(1,2C(1,0)) + C(3,2C(1,0))$. 

\medskip \noindent
By putting together our two inequalities, we conclude that
$$d(p_1,p_2) \leq d(a_1,a_2)+ d(b_1,b_2) \leq C(1,2C(1,0)) + C(3,2C(1,0)) +2C(1,0).$$
Thus, we get an upper bound on the diameter of $\gamma \cap Y$ that depends only on the Morse gauge $M$, as desired. 
\end{proof}

\noindent
In the context of quasi-median graphs endowed with local metrics, Proposition~\ref{prop:MorseInterProd} allows us to deduce that:

\begin{cor}\label{cor:DiamHypBounded}
Given a Morse gauge $M$ and a constant $D \geq 0$, there exists a constant $K \geq 0$ such that the following holds. Let $X$ be  a quasi-median graph and $\{(C,\delta_C) \mid C \text{ clique}\}$ a coherent system of graph-metrics. Assume that, for every clique $C$ of $X$:
\begin{itemize}
	\item if $(C,\delta_C)$ is bounded, then it has diameter $\leq D$;
	\item if $(C,\delta_C)$ is unbounded, then the fibres of the hyperplane containing $C$ are unbounded with respect to $\delta$. 
\end{itemize}
Then, $\mathrm{diam}(\gamma,\delta_J) \leq K$ for every hyperplane $J$ and every $M$-Morse geodesic $\gamma$ in $(X,\delta)$.
\end{cor}

\begin{proof}
Let $J$ be a hyperplane and $\gamma$ an $M$-Morse geodesic. We know from Lemma~\ref{lem:SystemGated} that $(N(J), \delta)$ is a convex subgraph of $(X,\delta)$. Therefore, the intersection the intersection between $\gamma$ and $(N(J),\delta)$ must be some (possibly infinite) subsegment $\gamma_J$. But we know from Theorem~\ref{thm:HypQM} that $N(J)$ decomposes as a product $C \times F$, where $C$ is a clique contained in $J$ and where $F$ is a fibre of $F$, so Lemma~\ref{lem:ProductSyst} implies that $(N(J) , \delta)$ decomposes as $(C, \delta_C) \times (F, \delta)$. If $(C,\delta_C)$ is bounded, then it has diameter $\leq D$, hence $\mathrm{diam}(\gamma, \delta_J) \leq D$. Otherwise, if $(C,\delta_C)$ is unbounded, then $(F,\delta)$ must be unbouded as well. Thus, we can apply Proposition~\ref{prop:MorseInterProd} and we deduce that $\gamma_0$ has length $\leq K_0$ for some constant $K_0$ that depends only on $M$. Since $\gamma$ does not cross $J$ outside $\gamma_0$, we conclude that 
$$\mathrm{diam}(\gamma ,\delta_J) = \mathrm{diam}(\gamma_J,\delta_J) = \mathrm{length}(\gamma_J) \leq K_0.$$
Thus, we have proved that $\mathrm{diam}(\gamma,\delta_J) \leq \max(D, K_0)$. 
\end{proof}

\noindent
By combining what we have proved so far, we are finally ready to prove Theorem~\ref{thm:CoarseSepGP}.

\begin{proof}[Proof of Theorem~\ref{thm:CoarseSepGP}.]
If $\Gamma$ is not connected, then $\Gamma \mathcal{G}$ and all the $\Gamma(\nu)$ are free products, so coarsely separated by bounded sets. Thus, there is nothing to prove in this case. From now one, we assume that $\Gamma$ is connected.

\medskip \noindent
We endow $\mathrm{QM}(\Gamma,\mathcal{G})$ with the canonical colouring $\kappa$ of its hyperplanes by $V(\Gamma)$ and with a system of metrics $\{(C,\delta_C) \mid C \text{ clique}\}$ such that $(\mathrm{QM}(\Gamma, \mathcal{G}),\delta)$ is (quasi-)isometric to $\Gamma \mathcal{G}$ (Proposition~\ref{prop:GPlocal}).

\medskip \noindent
Because $\Gamma$ is not a join and contains at least two vertices, we know from \cite{MR3368093} that $\Gamma \mathcal{G}$ is acylindrically hyperbolic. It follows from Theorem~\ref{thm:SeparatingAcylHyp} that, if $\Gamma \mathcal{G}$ is coarsely separated by a family $\mathcal{Z}$ of subexponential growth, then $\mathcal{Z}$ coarsely separates some Morse geodesic $\gamma$ of $\Gamma \mathcal{G}$. Let us verify that Proposition~\ref{prop:FindingInvasive} can be applied to $\gamma$. 

\medskip \noindent
Given an integer $k \geq 0$, let $\nu_k$ denote the weight function defined on $V(\Gamma)$ that is constant to $k$. Notice that, because all the groups in $\mathcal{G}$ are infinite, all the cliques in $\mathrm{QM}(\Gamma,\mathcal{G})$ are infinite, which implies that every hyperplane delimits infinitely many sectors. 

\medskip \noindent
Let $C$ be a clique of $\mathrm{QM}(\Gamma, \mathcal{G})$. According to Lemma~\ref{lem:GPcliqueQM}, it can be written as $gG_u$ for some $g \in \Gamma \mathcal{G}$ and $u \in V(\Gamma)$. Because $\Gamma$ is connected and contains at least two vertices, $u$ has at least one neighbour in $\Gamma$, say $v$. Notice that $g \langle G_u, G_v \rangle= g( G_u \oplus G_v)$ is a product of the two cliques $C$ and $D:=gG_v$, so Lemma~\ref{lem:ProductSyst} implies that $(P,\delta)$ is isometric to $(C,\delta_C) \times (D,\delta_D)$. This implies that the fibres of the hyperplane of $\mathrm{QM}(\Gamma,\mathcal{G})$ containing $C$ are unbounded with respect to $\delta$ since they contain copies of $(D,\delta_D) \simeq G_v$. Thus, Corollary~\ref{cor:DiamHypBounded} applies and shows that there exists a uniform constant $K$ such that $\mathrm{diam}(\gamma,\delta_J) \leq K$ for every hyperplane $J$ of $\mathrm{QM}(\Gamma,\mathcal{G})$.

\medskip \noindent
Thus, Proposition~\ref{prop:FindingInvasive} applies and shows that $\gamma$ is contained in a $\kappa$-regular invasive subgraph $Y$ of $\mathrm{QM}(\Gamma,\mathcal{G})$ such that $Y$ quasi-isometrically embeds into $(\mathrm{QM}(\Gamma, \mathcal{G}),\delta)$, i.e.\ in $\Gamma \mathcal{G}$, and such that every hyperplane $J$ crossing $Y$ has thickness $\geq k$. Of course, since $\mathcal{Z}$ coarsely separates $\gamma$, necessarily it coarsely separates $Y$. But we know from Proposition~\ref{prop:InvasiveSubIso} that $Y$ is (quasi-)isometric to $\Gamma(\nu)$ for some $\nu \geq k$. We conclude, as desired, that $\Gamma(\nu)$ is coarsely separated by some family of subexponential growth. 
\end{proof}

\section{Coarse separation in right-angled Artin groups}\label{section:RAAGs_thick}

\noindent
In this section, we prove Theorem~\ref{thm:BigIntro}. First, we introduce the notion of a \emph{unpinched} graph, and then, in Theorem~\ref{thm:clique_indecomp_graphs}, we prove a decomposition result for such graphs in the triangle-free case. This allows us to establish, in Theorem~\ref{thm:RAAGs_are_thick}, a thick decomposition of right-angled Artin groups whose defining graph is triangle-free and unpinched. Proposition~\ref{prop:reduction_triangle_free} then shows that, in order to prove Theorem~\ref{thm:BigIntro}, one may reduce to the triangle-free case. The thick decomposition of right-angled Artin groups further shows that it is enough to consider right-angled Artin groups defined by cycles of length at least~$4$. Finally, the proof of Theorem~\ref{thm:BigIntro} follows by combining Theorem~\ref{thm:CoarseSepGP} with the coarse-separation results for one-ended hyperbolic groups from \cite{BGT26hyp}.

\subsection{Unpinched graphs}\label{sec:unpinched_graphs}

\noindent
In this section, we introduce and study the following family of graphs:

\begin{definition}\label{def:clique_indecomposable_graph}
A graph $G$ is \emph{unpinched} if it is neither complete nor separated by a clique.
\end{definition}

\noindent
We emphasize that an unpinched graph must be connected, since otherwise it would be separated by the empty subgraph.

\medskip \noindent
A subgraph $H \leq G$ is called \emph{induced} if, for any pair of vertices $u,v \in V(H)$, the vertices $u$ and $v$ are adjacent in $H$ if and only if they are adjacent in $G$. If $W \subseteq V(G)$, we also denote by $W$ the subgraph of $G$ induced by the vertex-set $W$.

\begin{thm}\label{thm:clique_indecomp_graphs}
Let $G$ be a finite triangle-free unpinched graph. Then at least one of the following holds:
\begin{itemize}
    \item[(1)] There exists a vertex $v \in G$ such that $G \setminus \{v\}$ is unpinched.
    \item[(2)] There exists an induced subgraph $K \leq G$ such that 
    \begin{itemize} 
        \item[(i)] $K$ separates $G$; 
        \item[(ii)] there exists a partition $G \setminus K = \sqcup_{i = 1}^p G_i$, and the subgraph induced by $G_i \cup K$ is unpinched for every $i$. 
    \end{itemize}
    \item[(3)] $G$ is a cycle of length $\geq 4$.
\end{itemize}
\end{thm}

\begin{figure}[htbp]
  \centering
  \includegraphics[width=0.8\textwidth]{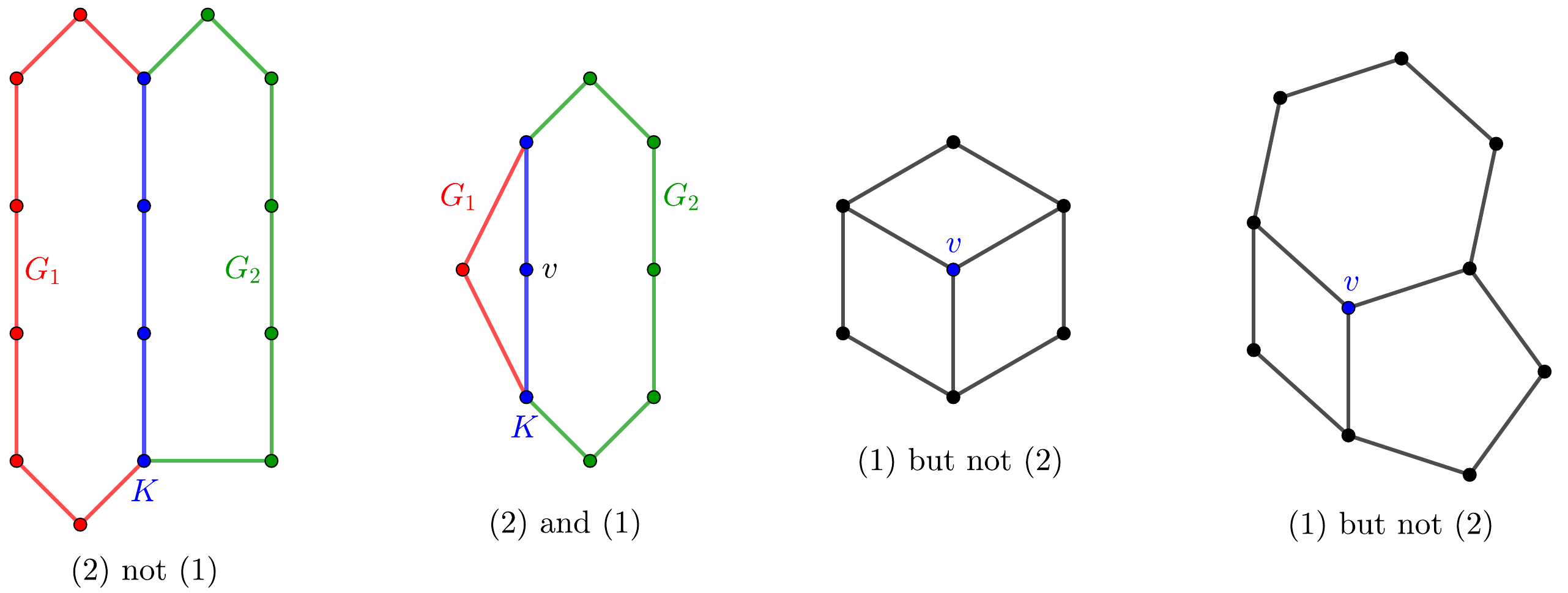}
  \caption{Examples satisfying $(1)$, $(2)$, or both.}
  \label{figure examples 1,2}
\end{figure}

\noindent
Note that in an unpinched graph $G$, the link of any vertex $v$ is not a clique, since otherwise it would be a separating clique or $G$ would be complete. Note also that the triangle-free assumption in \Cref{thm:clique_indecomp_graphs} is necessary, as shown by the counterexamples in Figure~\ref{figure Counterexamples in the presence of triangles}.
\begin{figure}[htbp]
  \centering
  \includegraphics[width=0.5\textwidth]{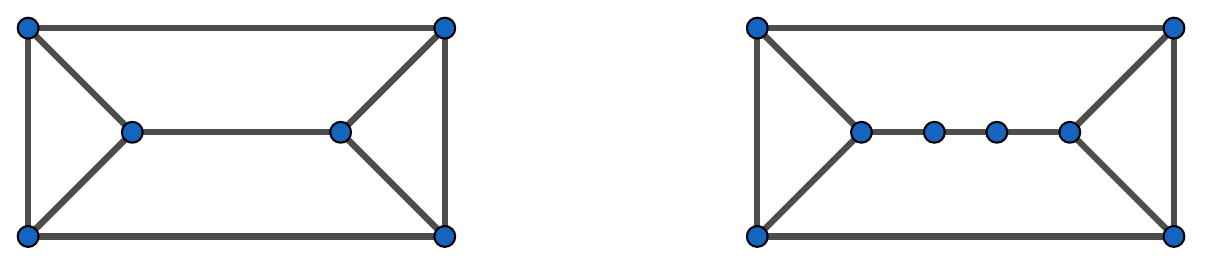}
  \caption{Counterexamples in the presence of triangles.}
  \label{figure Counterexamples in the presence of triangles}
\end{figure}

\medskip \noindent
Let us first prove some useful lemmas.

\begin{lemma}\label{lem:separating_clique_indecomp_subgraph}
    Let $G$ be an unpinched graph, and $Q \leq G$ an induced unpinched subgraph. If $Q$ separates $G$, then $G$ satisfies $(2)$ of \Cref{thm:clique_indecomp_graphs}.
\end{lemma}
\begin{proof}
    Let $G_1, \dots, G_p$ be the connected components of $G \setminus Q$. We claim that, for each $i$, the graph $G_i \cup Q$ is unpinched. It is clearly connected and not complete. Let $K \leq G_i \cup Q$ be a complete subgraph, and let $x,y \in (G_i \cup Q) \setminus K$.
\begin{itemize}
    \item If $x,y \in Q$, then $Q \cap K$ is a complete subgraph in $Q$, and since $Q$ is unpinched, there exists a path in $Q$ from $x$ to $y$ avoiding $Q \cap K$ (hence avoiding $K$).

    \item If $x,y \in G_i$, then $K$ does not separate $x$ or $y$ from $Q$, otherwise it would separate them from the other connected components of $G \setminus Q$. So there exist $x',y' \in Q$ and paths $x \to x'$, $y \to y'$ avoiding $K$. As in the previous case, $Q \cap K$ is a complete subgraph in $Q$, and $Q$ is unpinched, so there is a path in $Q$ from $x'$ to $y'$ avoiding $Q \cap K$. Putting the paths together yields a path $x \to y$ in $G_i \cup Q$ avoiding $K$.

    \item If $x \in Q$ and $y \in G_i$, similarly, $K$ does not separate $y$ from $Q$, so there exists $y' \in Q$ and a path $y \to y'$ avoiding $K$, and there is a path $x \to y'$ in $Q$ avoiding $K$. \qedhere
\end{itemize}
\end{proof}

\noindent
If $\Gamma$ is a graph and $v \in \Gamma$ is a vertex, we denote by $N(v)$ the set of neighbours of $v$.

\begin{lemma}\label{lemma:adding_path_to_clique_indecomp}
    Let $\Gamma$ be a graph, and let $G \leq \Gamma$ be an induced unpinched subgraph.
    \begin{enumerate}
        \item If $\gamma \leq \Gamma$ is an induced path whose endpoints $a,b \in G$ are non-adjacent and whose interior vertices are disjoint from $G$, then the subgraph induced by $G$ and $\gamma$ is unpinched.
        \item If $v \in \Gamma \setminus G$ is such that $N(v) \cap G$ is not complete, then the subgraph induced by $G$ and $v$ is unpinched.
    \end{enumerate}
\end{lemma}

\noindent
The proof is similar to that of Lemma~\ref{lem:separating_clique_indecomp_subgraph}. We include it for completeness.
\begin{proof}
    $(1)$ Let $G'$ denote the subgraph induced by $G$ and $\gamma$. It is clearly connected and not complete. Let $K \leq G'$ be a complete subgraph, and let $x,y \in G' \setminus K$. 
    \begin{itemize}
        \item If $x,y \in G$, then $G \cap K$ is a complete subgraph in $G$, and since $G$ is unpinched, there is a path in $G$ from $x$ to $y$ avoiding $G \cap K$, hence avoiding $K$.
        \item If $x,y \in \gamma \setminus G$, then each of $x$ and $y$ can be joined in $\gamma$ to at least one of $a$ or $b$ by a path avoiding $K$. We then complete with a path in $G$, as in the previous case.
        \item The case $x \in G$ and $y \in \gamma \setminus G$ is treated similarly. \qedhere
    \end{itemize}
\noindent
To prove $(2)$, let $G'$ be the subgraph induced by $G$ and $v$. Since $N(v)\cap G$ is not complete, there exist two non-adjacent neighbours $a,b \in G$ of $v$. Then $a \to v \to b$ is an induced path whose endpoints lie in $G$, so $(2)$ follows from $(1)$.
\end{proof}

\begin{proof}[Proof of \Cref{thm:clique_indecomp_graphs}]
    Let $G$ be a finite triangle-free unpinched graph. Suppose that $G$ is not a cycle. We will show that $(1)$ or $(2)$ holds.

    \noindent
    Since $G$ is unpinched, it is connected and not complete. Moreover, $G$ is not a tree, since every finite tree that is not complete is separated by a vertex. Hence $G$ contains a cycle, and therefore an induced cycle $Q \leq G$. Since $G$ is triangle-free, $Q$ has length at least $4$, so $Q$ is unpinched. Since $G$ is not a cycle, we have $Q \neq G$. As $G$ is finite, there exists a maximal proper induced unpinched subgraph $Q' \leq G$ containing $Q$.
    
    \noindent
    If $Q'$ separates $G$, then $G$ satisfies $(2)$ by Lemma~\ref{lem:separating_clique_indecomp_subgraph}. From now on, assume that $G \setminus Q'$ is connected. Let
    $$
    Z=\{z\in Q' \mid z \textup{ is adjacent to a vertex in } G\setminus Q'\}.
    $$
    Since $G$ is connected and $Q'\neq G$, the set $Z$ is non-empty. Moreover, $Z$ is not complete, otherwise it would separate $G$, contradicting the fact that $G$ is unpinched. Therefore, there exist non-adjacent vertices $q_1,q_2 \in Z$. Let
    $$
    A_1=N(q_1)\cap (G\setminus Q') \qquad \textup{and} \qquad A_2=N(q_2)\cap (G\setminus Q'),
    $$
    where $N(q_i)$ denotes the set of neighbours of $q_i$ in $G$. Let $\gamma$ be a shortest path in $G\setminus Q'$ from $A_1$ to $A_2$, with endpoints $x_1\in A_1$ and $x_2\in A_2$. Then $\gamma$ is induced in $G\setminus Q'$.
    \begin{claim}\label{claim:complement_is_path}
        $G \setminus Q' = \gamma$.
    \end{claim}
    \begin{proof}[Proof of the claim]
        Let $\widetilde{\gamma}$ denote the subgraph induced by $\gamma \cup \{q_1,q_2\}$. Since $\gamma$ is induced in $G\setminus Q'$, it follows that $\widetilde{\gamma}$ is an induced path (in $G$): indeed, if for instance $q_1$ were adjacent to an interior vertex $w$ of $\gamma$, then $w \in A_1$, so the subpath of $\gamma$ from $w$ to $x_2$ would be a path in $G\setminus Q'$ from $A_1$ to $A_2$ shorter than $\gamma$, a contradiction. Similarly, $q_2$ is not adjacent to any interior vertex of $\gamma$. If $ \gamma \subsetneq G \setminus Q'$, then the subgraph induced by $Q' \cup \widetilde{\gamma}$ is unpinched by Lemma \ref{lemma:adding_path_to_clique_indecomp} and $\ne G$, which contradicts again the maximality of $Q'$. Therefore $G\setminus Q'=\gamma$.
    \end{proof}
    \begin{claim}\label{claim:interior_path}
    \begin{itemize}
        \item[(i)] $x_1$ (resp.\ $x_2$) has exactly one neighbour in $Q'$, namely $q_1$ (resp.\ $q_2$).
        \item[(ii)] The interior of $\gamma$ has at most one neighbour in $Q'$, which, if it exists, is distinct from both $q_1$ and $q_2$.
    \end{itemize}
    \end{claim}
    
    \begin{proof}[Proof of the claim]
        \noindent 
        $(i)$ Suppose that $x_1$ has another neighbour $q_1' \in Q'$. Since $G$ is triangle-free, $q_1$ and $q_1'$ are not adjacent. Hence the path $q_1 \to x_1 \to q_1'$ is induced, and by Lemma~\ref{lemma:adding_path_to_clique_indecomp}, the subgraph induced by $Q' \cup \{x_1\}$ is unpinched, contradicting the maximality of $Q'$. The argument for $x_2$ is identical.
        
        \medskip\noindent
        $(ii)$ First, no interior vertex of $\gamma$ is adjacent to $q_1$ or $q_2$: otherwise, as in the proof of \Cref{claim:complement_is_path}, one obtains a shorter path from $A_1$ to $A_2$ in $G\setminus Q'$.
        
        \noindent
        Now let $w$ be an interior vertex of $\gamma$ adjacent to some $p \in Q'$. We claim that $p$ is adjacent to both $q_1$ and $q_2$. Suppose for instance that $p \nsim q_1$, and choose $w$ as close as possible to $x_1$ with this property. Then the path $q_1 \to x_1 \to \dots \to w \to p$ is induced, so by Lemma~\ref{lemma:adding_path_to_clique_indecomp}, the subgraph induced by $Q' \cup \{x_1,\dots,w\}$ is unpinched, contradicting the maximality of $Q'$. Similarly, $p$ is adjacent to $q_2$.
        
        \noindent
        Finally, suppose that two distinct interior vertices $w,w'$ of $\gamma$ are adjacent to vertices $p,p' \in Q'$. Then $p$ and $p'$ are both adjacent to $q_1$, hence are non-adjacent since $G$ is triangle-free. Choosing $w,w'$ so that the subpath between them is minimal, the path $p \to w \to \dots \to w' \to p'$ is induced. Lemma~\ref{lemma:adding_path_to_clique_indecomp} again yields a contradiction with the maximality of $Q'$. Therefore, the interior of $\gamma$ has at most one neighbour in $Q'$, and this neighbour is distinct from both $q_1$ and $q_2$.
    \end{proof}

    \noindent
    Let us treat the two cases of Claim~\ref{claim:interior_path}$(ii)$ separately.
    
    \medskip\noindent
    \textbf{Case 1}: no vertex in the interior of $\gamma$ is adjacent to $Q'$. Let $\gamma'$ be an induced path in $Q'$ from $q_1$ to $q_2$. Then $\gamma'$ separates $G$. Indeed, by Claim~\ref{claim:interior_path}, the only edges between $\gamma$ and $Q'$ are $x_1q_1$ and $x_2q_2$, and both $q_1$ and $q_2$ belong to $\gamma'$. Let
    $$
    G_1=\widetilde{\gamma}\setminus \gamma'=\gamma \qquad \textup{and} \qquad G_2=Q'\setminus \gamma',
    $$
    where $\widetilde{\gamma}$ denotes the subgraph induced by $\gamma \cup \{q_1,q_2\}$. Then $G_2\cup \gamma'=Q'$ is unpinched. Moreover, $G_1\cup \gamma'=\widetilde{\gamma}\cup \gamma'$ is an induced cycle, since there is no edge between $\gamma$ and $\gamma'$ other than $x_1q_1$ and $x_2q_2$. Its length is at least $4$, so it is also unpinched. Therefore $G$ satisfies $(2)$.
    
    \begin{figure}[htbp]
    \centering
    \includegraphics[width=0.35\textwidth]{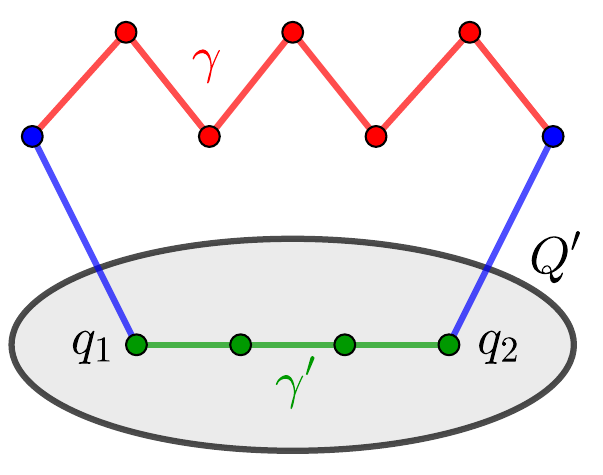}
    \caption{Case 1.}
    \label{figure case 1}
    \end{figure}

    \medskip\noindent
    \textbf{Case $2$}: The interior of $\gamma$ is adjacent to $Q'$. Let $p \in Q'$ be the unique neighbour of the interior of $\gamma$ as in Claim~\ref{claim:interior_path}. Since $Q'$ is unpinched, it is not separated by a vertex. Let $\gamma'$ be a shortest path in $Q'$ from $q_1$ to $q_2$ avoiding $p$. Then $\gamma'$ is induced. Let $K$ be the subgraph of $Q'$ induced by $\gamma' \cup \{p\}$. We distinguish two subcases.
    
    \smallskip\noindent
    - If $K = Q'$, then $G \setminus \{p\}$ is the union of $\widetilde{\gamma}$ and $\gamma'$, hence an induced cycle: indeed, $\gamma'$ avoids $p$, and by Claim~\ref{claim:interior_path}, the only edges between $\gamma$ and $Q'$ are $x_1q_1$, $x_2q_2$, and the edges from the interior of $\gamma$ to $p$. Thus, after removing $p$, the only edges between $\widetilde{\gamma}$ and $\gamma'$ are $x_1q_1$ and $x_2q_2$. Since $q_1$ and $q_2$ are non-adjacent, this cycle has length at least $4$, so $G \setminus \{p\}$ is unpinched. Hence $(1)$ holds.

    \smallskip\noindent
    - If $K \subsetneq Q'$, let $G_1=\gamma$ and $G_2=Q'\setminus K$. Then $K$ separates $G$, and
    $$
    G \setminus K = G_1 \sqcup G_2.
    $$
    Moreover, $G_2 \cup K = Q'$ is unpinched. Note that
    $$
    G_1 \cup K=\gamma \cup \gamma' \cup \{p\}.
    $$
    The subgraph induced by $\gamma$ and $\gamma'$ is a cycle, so it is unpinched by the same argument as before. Also, $p$ is adjacent to $q_1$ and $q_2$, so $N(p)\cap (\gamma \cup \gamma')$ is not a clique. Therefore, by the second item of Lemma~\ref{lemma:adding_path_to_clique_indecomp}, the subgraph $G_1 \cup K$ is unpinched. Hence $(2)$ is satisfied.
    \begin{figure}[htbp]
  \centering
  \includegraphics[width=0.4\textwidth]{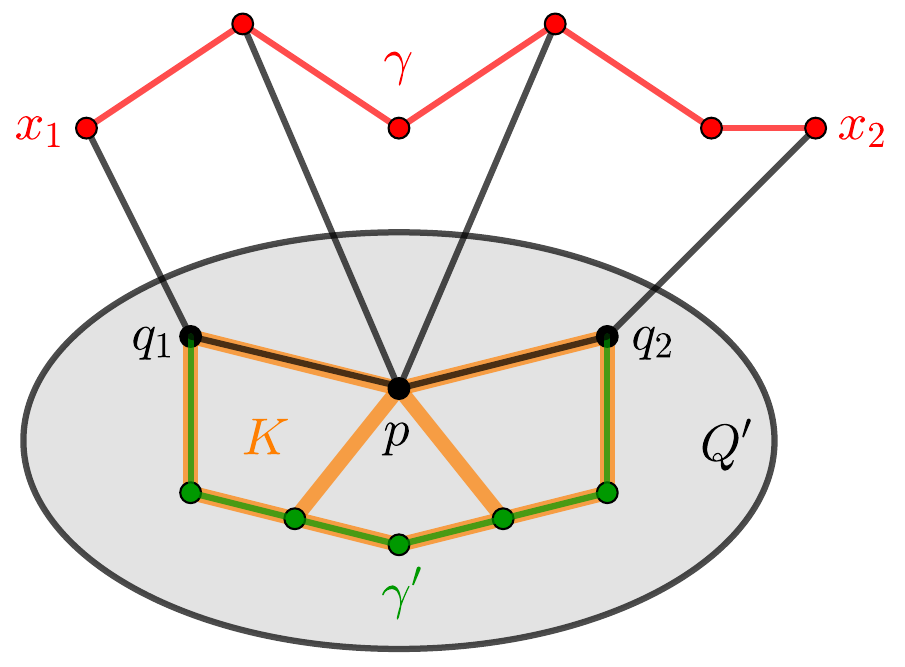}
  \caption{Case 2.}
  \label{figure case 2}
\end{figure}
\end{proof}

\subsection{Thick decomposition of right-angled Artin groups}\label{sec:thickdecomp}

\noindent
Using Theorem~\ref{thm:clique_indecomp_graphs}, we now prove the following thick decomposition result for right-angled Artin groups.

\begin{thm}\label{thm:RAAGs_are_thick}
Let $\Gamma$ be a finite triangle-free unpinched graph, and denote $r=|V(\Gamma)|$. Let $A(\Gamma)$ be the associated right-angled Artin group, equipped with the word metric with respect to the standard generating set. Let $\mathcal{A}$ be the family of $r$-neighborhoods of left cosets of the parabolic subgroups induced by cycles in $\Gamma$, and let $\mathcal{B}$ be the family of left cosets of the parabolic subgroups induced by pairs of non-adjacent vertices in $\Gamma$. Then $A(\Gamma)$ is $(\mathcal{A},\mathcal{B})$-thick.
\end{thm}

\begin{proof}[Proof of \Cref{thm:RAAGs_are_thick}]
     Throughout the proof, all right-angled Artin groups are equipped with the word metric with respect to their standard generating set. We argue by induction on $n:=|V(\Gamma)|$. There are no unpinched graphs with at most $3$ vertices. The only unpinched graph on $4$ vertices is the square, so there is nothing to prove. Let $\Gamma$ be a triangle-free, unpinched graph on $n \geq 5$ vertices. By \Cref{thm:clique_indecomp_graphs}, we distinguish three cases.

\medskip\noindent
\textbf{$\Gamma$ is a cycle:} there is nothing to prove.

\medskip\noindent
\textbf{$\Gamma$ satisfies $(1)$ of \Cref{thm:clique_indecomp_graphs}:} there exists a vertex $v \in \Gamma$ such that the subgraph induced by $\Lambda := \Gamma \setminus \{v\}$ is unpinched. By the induction hypothesis, $H:=A(\Lambda) \leq A(\Gamma)$ is $(\mathcal{A},\mathcal{B})$-thick, where $\mathcal{A}$ is the family of $(n-1)$-neighborhoods of left cosets (in $H$) of the parabolic subgroups induced by cycles in $\Lambda$, and $\mathcal{B}$ is the family of left cosets (in $H$) of the parabolic subgroups induced by pairs of non-adjacent vertices in $\Lambda$. Enlarge both families by allowing left translations by elements of $A(\Gamma)$. Therefore, every left coset of $H$ is $(\mathcal{A},\mathcal{B})$-thick.

\medskip \noindent
Since $\Gamma$ is unpinched, $v$ admits two non-adjacent neighbors $a,b$ in $\Lambda$. Let $\mathcal C$ be the family of $1$-neighborhoods of left cosets of $H$, and let $\mathcal D$ be the family of left cosets of the subgroup $\langle a,b \rangle$. By \Cref{lem:thick_family_left_cosets}, $A(\Gamma)$ is $(\mathcal{C},\mathcal{D})$-thick. Indeed, $ H \cap vHv^{-1}$ contains $\langle a,b \rangle$. $\mathcal D$ is a subfamily of $\mathcal B$, so $A(\Gamma)$ is $(\mathcal{C},\mathcal{B})$-thick. Moreover, by the previous paragraph, every left coset $gH$ is $(\mathcal{A},\mathcal{B})$-thick, so by \Cref{lem:thick_neighborhood}, $gH^{+1}$ is $(\mathcal{A}^{+1},\mathcal{B})$-thick. Therefore, every element of $\mathcal C$ is $(\mathcal{A}^{+1},\mathcal{B})$-thick, and by the transitivity of \Cref{lem:thickness_transitivity}, we get that $A(\Gamma)$ is $(\mathcal{A}^{+1},\mathcal{B})$-thick. Since $\mathcal{A}$ is the family of $(n-1)$-neighborhoods of left cosets (in $H$) of the parabolic subgroups induced by cycles in $\Lambda$, we are done.
    
\medskip \noindent
\textbf{$\Gamma$ satisfies $(2)$ of Theorem~\ref{thm:clique_indecomp_graphs}:} there exists an induced subgraph $K \leq \Gamma$ such that $K$ separates $\Gamma$, and there exists a partition
$$
\Gamma \setminus K = \bigsqcup_{i=1}^p \Gamma_i,
$$
such that, for each $i$, the subgraph induced by $\Gamma_i \cup K$, which we denote by $\Lambda_i$, is unpinched. In particular, $A(\Gamma)$ splits as an amalgamated free product over $A(K)$ with vertex-groups $A(\Lambda_i)$. Hence $A(\Gamma)$ is $(\mathcal{C},\mathcal{D})$-thick, where $\mathcal{C}$ is the family of left cosets of the subgroups $A(\Lambda_i)$, for $1 \leq i \leq p$, and $\mathcal{D}$ is the family of left cosets of $A(K)$, see \Cref{ex:thick_spaces}.

\medskip \noindent
Let $\mathcal{A}$ and $\mathcal{B}$ be the families from the statement of the theorem. Since $K$ is not complete, it contains two non-adjacent vertices, so every $D \in \mathcal{D}$ contains some $B \in \mathcal{B}$. Therefore $A(\Gamma)$ is $(\mathcal{C},\mathcal{B})$-thick. Moreover, by the induction hypothesis, every $C \in \mathcal{C}$ is $(\mathcal{A},\mathcal{B})$-thick. Therefore, by \Cref{lem:thickness_transitivity}, we conclude that $A(\Gamma)$ is $(\mathcal{A},\mathcal{B})$-thick.
\end{proof}

\subsection{Reduction to the triangle-free case}\label{sec:trianglefreeReduc}

\noindent
In this section, we show that, to prove Theorem~\ref{thm:BigIntro}, it suffices to consider right-angled Artin groups defined by triangle-free graphs. More precisely, we prove the following.

\begin{prop}\label{prop:reduction_triangle_free}
    Let $\Gamma$ be a finite unpinched graph. If $A(\Gamma)$ admits a coarsely separating family of subexponential growth, then so does $A(\Gamma')$ for some finite triangle-free unpinched graph $\Gamma'$.
\end{prop}

\begin{proof}
    Let $v \in \Gamma$ be a vertex contained in a triangle, and let $L$ denote its link. For each pair of non-adjacent vertices $x,y \in L$, consider a shortest path (hence induced) from $x$ to $y$ in $\Gamma \setminus \{v\}$. Such a path exists since $\Gamma$ is unpinched, hence in particular is not separated by a single vertex. Let these paths be denoted by $\alpha_1,\dots,\alpha_N$.

\noindent
    \begin{figure}[h!]
  \centering
  \includegraphics[width=0.2\textwidth]{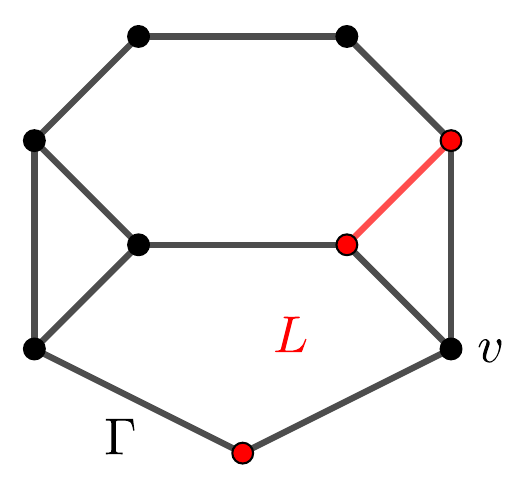}
  \caption{$\Gamma$ and link of $v$}
  \label{figure Gamma and link v}
\end{figure}  
\begin{figure}[h!]
  \centering
  \includegraphics[width=0.25\textwidth]{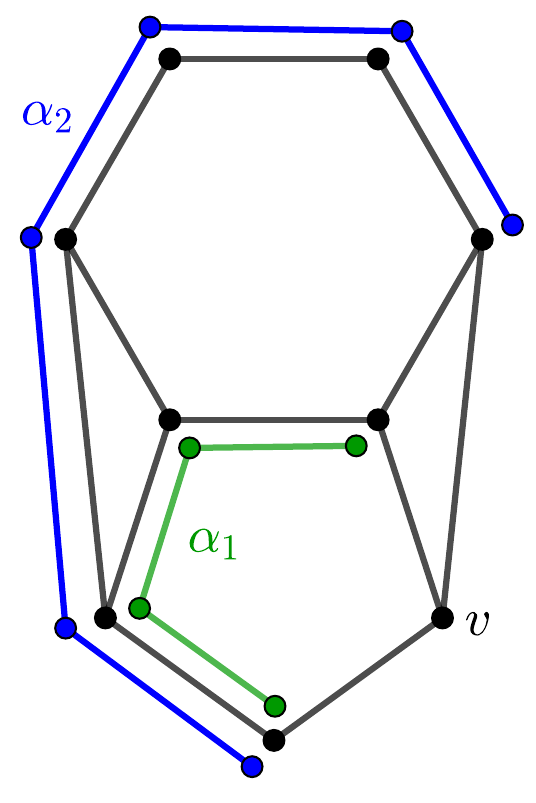}
  \caption{the paths $\alpha_i$}
  \label{figure the paths}
\end{figure}  

\noindent
    We consider $2N$ copies of $\Gamma$, denoted $\Gamma_1,\dots,\Gamma_N$ and $\Gamma_1',\dots,\Gamma_N'$. We glue all these copies along the star of $v$, and denote the resulting graph by $\widetilde{\Gamma}$. For each $i=1,\dots,N$, let $\gamma_i$ (resp.\ $\gamma_i'$) be the path corresponding to $\alpha_i$ in $\Gamma_i$ (resp.\ $\Gamma_i'$). Let $K \leq \widetilde{\Gamma}$ denote the subgraph induced by $\{v\} \cup L$ together with these $2N$ paths.

    \medskip \noindent
    Note that $A(\Gamma)$ and $A(\widetilde{\Gamma})$ are quasi-isometric \cite[Section~11]{MR2421136}. Hence, if $A(\Gamma)$ admits a coarsely separating subset of subexponential growth, then so does $A(\widetilde{\Gamma})$ by \cite[Lemma~2.3]{bensaid2024coarse}.

    \medskip \noindent
    Since the subgraph $K$ separates $\widetilde{\Gamma}$, it follows that $A(\widetilde{\Gamma})$ decomposes as an amalgamated free product over the subgroup $A(K)$ with vertex-groups $A(\Gamma_i \cup K)$ and $A(\Gamma_i' \cup K)$. Moreover, $A(K)$ has exponential growth since $K$ is not complete. Therefore, by \Cref{prop:Thick} and \Cref{prop:coarse_sep_family_converse} (see also the proof of \Cref{cor:coarse_sep_amalgams}), some left coset of one of the vertex-groups is coarsely separated by a subset of subexponential growth.
    \begin{claim}\label{claim vertex groups thick and have less triangles}
    For every $i$, the graphs $\Gamma_i \cup K$ and $\Gamma_i' \cup K$ are unpinched.
    \end{claim}
    \begin{proof}
    For every $i$, both $\Gamma_i \cup K$ and $\Gamma_i' \cup K$ are obtained from a copy of $\Gamma$ by gluing all but one of the paths. Since the endpoints of each path are glued to non-adjacent vertices, adding each path preserves being unpinched by Lemma~\ref{lemma:adding_path_to_clique_indecomp}.
    \end{proof}
    \begin{figure}[htbp]
  \centering
  \includegraphics[width=0.3\textwidth]{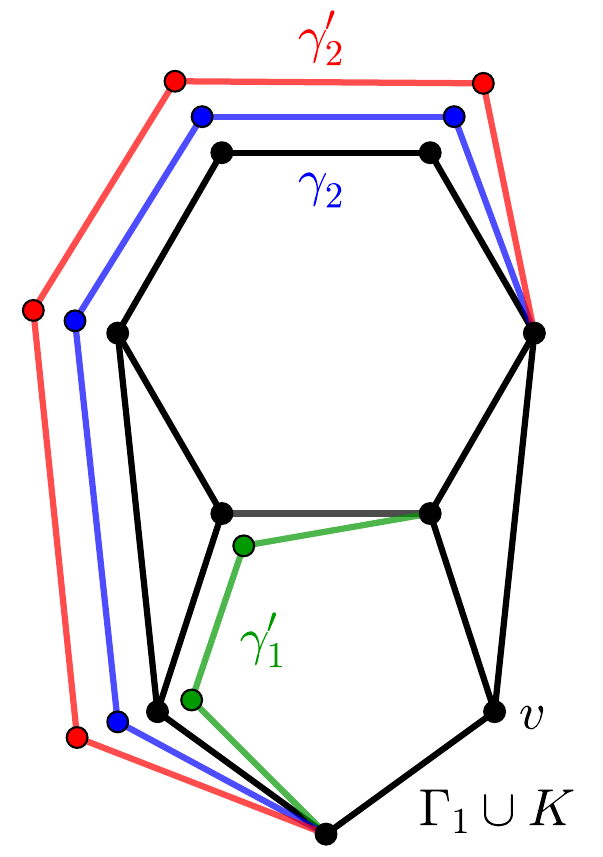}
  \caption{$\Gamma_1 \cup K$}
  \label{figure Gamma1 +K}
\end{figure}  
    \begin{claim}\label{claim removing v keeps graph thick}
    For every $i$, both $(\Gamma_i \cup K) \setminus \{v\}$ and $(\Gamma_i' \cup K) \setminus \{v\}$ are unpinched.
    \end{claim}
    \begin{proof}
        The same argument applies to both graphs, so let us consider the graph $(\Gamma_i \cup K) \setminus \{v\}$. Let us just denote $\Gamma_i$ by $\Gamma$. The graph $(\Gamma \cup K) \setminus \{v\}$ is obtained from a copy of $\Gamma \setminus \{v\}$ by gluing all but one of the paths. For convenience, we denote these paths by $\gamma_1, \dots, \gamma_p$, where $p = 2N - 1$. The graph $(\Gamma \cup K) \setminus \{v\}$ is clearly not complete. It is connected because $\Gamma$ is unpinched, hence not separated by the vertex $v$, so $\Gamma \setminus \{v\}$ is connected.

        \medskip \noindent
        Let $C \leq (\Gamma \cup K) \setminus \{v\}$ be a complete subgraph. Since $C$ has diameter $\leq 2$, it is either contained in some $\gamma_j$, or contained in $\Gamma \setminus \{v\}$. If $C$ is contained in some $\gamma_j$, then it is clearly not separating.

        \medskip \noindent
        Suppose that $C$ is contained in $\Gamma \setminus \{v\}$, and let $x, y \in \left((\Gamma \cup K) \setminus \{v\}\right) \setminus C$. Since the endpoints of each $\gamma_j$ are non-adjacent, they cannot lie in the same complete subgraph. Thus, any interior point of $\gamma_j$ is connected to $\Gamma \setminus \{v\}$ by a path that avoids $C$, so it suffices to treat the case where both $x$ and $y$ lie in $\Gamma \setminus \{v\}$. Since $C$ does not separate $\Gamma$, there exists a shortest (and hence induced) path from $x$ to $y$ in $\Gamma$ that avoids $C$. If this path passes through $v$, let $a$ and $b$ be the neighbours of $v$ along the path. Since the path is induced, $a$ and $b$ are not adjacent. Hence, one of the glued paths, say $\gamma_k$, joins $a$ to $b$. Since $C \subset \Gamma \setminus \{v\}$ and $\gamma_k$ lies outside $\Gamma \setminus \{v\}$, except at its endpoints $a$ and $b$, the path $\gamma_k$ avoids $C$. Replacing the subpath $a \to v \to b$ with $\gamma_k$ therefore yields a path from $x$ to $y$ avoiding $C$.
    \end{proof}
    
    \begin{figure}[htbp]
    \centering
    \includegraphics[width=0.35\textwidth]{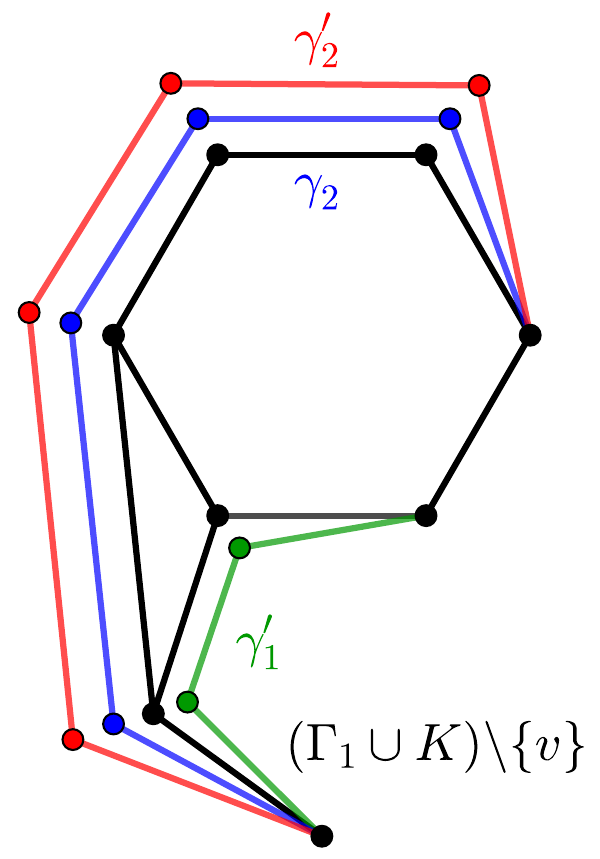}
    \caption{$(\Gamma_1 \cup K) \setminus \{v\}$}
    \label{figure Gamma1 +K minus v}
    \end{figure}  
    
    \begin{claim}\label{claim nb triangles decreasing}
    For every $i$, both $(\Gamma_i \cup K) \setminus \{v\}$ and $(\Gamma_i' \cup K) \setminus \{v\}$ contain strictly fewer triangles than $\Gamma$.
    \end{claim}
    \begin{proof}
    For every $i$, both $\Gamma_i \cup K$ and $\Gamma_i' \cup K$ are obtained from a copy of $\Gamma$ by gluing all but one of the paths. Since each added path is induced and its endpoints are non-adjacent, this operation creates no new triangle. Hence both graphs contain exactly the same number of triangles as $\Gamma$. Removing $v$ then deletes all triangles containing it. Since $v$ was chosen to lie in at least one triangle, the total number of triangles strictly decreases.
    \end{proof}

    \noindent
    Finally, let $\Lambda$ be $\Gamma_i \cup K$ or $\Gamma_i' \cup K$ for some $i$. Since both $\Lambda$ and $\Lambda \setminus \{v\}$ are unpinched, $v$ admits two non-adjacent neighbors $a,b$ in $\Lambda \setminus \{v\}$. Let $\mathcal A$ be the family of $1$-neighborhoods of left cosets of $A(\Lambda \setminus \{v\})$, and let $\mathcal B$ be the family of left cosets of the subgroup $\langle a,b \rangle$. By \Cref{lem:thick_family_left_cosets}, $A(\Lambda)$ is $(\mathcal{A},\mathcal{B})$-thick. Indeed, $ A(\Lambda \setminus \{v\}) \cap v  A(\Lambda \setminus \{v\})  v^{-1}$ contains $\langle a,b \rangle$. Therefore, by \Cref{prop:Thick} and \Cref{prop:coarse_sep_family_converse}, if a family of subexponential growth coarsely separates $A(\Lambda)$, it must coarsely separate the $1$-neighborhood of some left coset of $A(\Lambda\setminus \{v\})$. Therefore, some some left coset of $A(\Lambda\setminus \{v\})$ is coarsely separated by a family of subexponential growth.  If $\Lambda\setminus \{v\}$ is triangle-free, we are done. Otherwise, we repeat the process with a vertex $u$ that lies in a triangle. Since the number of triangles strictly decreases at each step, the process terminates after finitely many steps.
\end{proof}

\subsection{Proof of Theorem~\ref{thm:BigIntro}}

\noindent
In this section, we combine all the results we have obtained so far in order to prove the main result of our article, namely Theorem~\ref{thm:BigIntro}. 

\begin{proof}[Proof of Theorem~\ref{thm:BigIntro}.]
The implications $(iii) \Rightarrow (ii) \Rightarrow (i)$ are clear. Our goal is to prove $(i) \Rightarrow (iii)$. Assume for contradiction that there exists a graph $\Gamma$ that is neither complete nor separable by a complete subgraph but such that $A(\Gamma)$ is coarsely separable by a family of subexponential growth. As a consequence of Proposition~\ref{prop:reduction_triangle_free}, we can assume that $\Gamma$ is triangle-free. Then, we know from Theorem~\ref{thm:RAAGs_are_thick} that $A(\Gamma)$ is $(\mathcal{A},\mathcal{B})$-thick where $\mathcal{A}$ is a collection of neighbourhoods of cosets of parabolic subgroups defined by cycles in $\Gamma$ and where $\mathcal{B}$ is a collection of neighbourhoods of free subgroups. If $\mathcal{Z}$ denotes a family of subexponential growth coarsely separating $A(\Gamma)$, it follows from Proposition~\ref{prop:Thick} that $\mathcal{Z}$ coarsely separates $\mathcal{A}$, and then from Proposition~\ref{prop:coarse_sep_family_converse} that $\mathcal{Z}$ coarsely separates some member of $\mathcal{A}$.

\medskip \noindent
Thus, we have prove that there exists some $n \geq 4$ such that $A(C_n)$ is coarsely separable by a family of subexponential growth, where $C_n$ denotes the cycle of length $n$. According to Theorem~\ref{thm:CoarseSepGP}, one can find a weight $\nu \geq 3$ such that the graph product $C_n(\nu)$ is coarsely separable by a family of subexponential growth.  
Notice that $C_4(\nu)$ cannot be coarsely coarsely by a family of subexponential growth. Indeed, $C_4(\nu)$ decomposes as $(A \ast B) \oplus (C \ast D)$ for some finite groups $A,B,C,D$ of cardinality $\geq 3$, and consequently is quasi-isometric to $\mathbb{F}_2 \times \mathbb{F}_2$. But we know from \cite[Theorem~1.3]{bensaid2024coarse} that $\mathbb{F}_2 \times \mathbb{F}_2$ cannot be coarsely separated by a family of subexponential growth. Therefore, we must have $n \geq 5$. This implies that $C_n(\nu)$ is hyperbolic (see \cite{GPHyp} (and also \cite[Section~8.3]{QM})). Then, \cite{BGT26hyp} implies that $C_n(\nu)$ is either multi-ended, or virtually a surface group, or splittable over a two-ended group. But we know from \cite{GPends} that $C_n(\nu)$ is one-ended, from \cite[Proposition~4.5]{BGT26hyp} that it is not virtually a surface group, and from \cite[Corollary~4.10]{BGT26hyp} that it does not split over a two-ended group. Hence the desired contradiction.
\end{proof}

\section{Other applications}

\noindent
In this final section, we record some applications of Theorem~\ref{thm:BigIntro} and of the techniques we have used to prove it.

\subsection{Complete-cut-decompositions}\label{section:JSJ}

\noindent
As a consequence of Theorem~\ref{thm:BigIntro}, there exist strong restrictions on possible coarse embeddings between right-angled Artin groups. In order to motivate this assertion, we need the following definition:

\begin{definition}
A \emph{complete-cut-decomposition} $(T,(V_s)_{s \in V(T)})$ of a graph $X$ is the data of a tree $T$ and a collection of induced subgraphs $V_s \leq X$ indexed by $V(T)$ such that:
\begin{itemize}
	\item $\{V_s \mid s \in V(T)\}$ covers $X$, i.e.\ $E(X)= \bigcup_{s \in V(T)} E(V_s)$;
	\item for every $s \in V(T)$, $V_s$ has no complete cut;
	\item for all adjacent $r,s \in V(T)$, $V_r \cap V_s$ is a complete cut of $X$ properly contained in both $V_r$ and $V_s$.
\end{itemize}
\end{definition}

\noindent
According to \cite{CodimRAAG}, every finite graph admits a complete-cut-decomposition. Then:

\begin{cor}\label{cor:RAAGJSJ}
Let $\Phi,\Psi$ be two finite graphs. Fix two complete-cut-decompositions $(R, (U_s)_{s \in V(R)})$ and $(S, (V_s)_{s \in V(S)})$ respectively of $\Phi$ and $\Psi$. If there exists a coarse embedding $\varphi : A(\Phi) \to A(\Psi)$, then, for every $r \in V(R)$ such that $U_r$ is not complete, there exists some $s \in V(S)$ such that $\varphi$ sends $\langle U_r \rangle$ in a neighbourhood of a coset of $\langle V_s \rangle$. 
\end{cor}

\begin{proof}
On the one hand, because $U_r$ is not complete and has no complete cut, it follows from Theorem~\ref{thm:BigIntro} that $\langle U_r \rangle$ is not coarsely separable by a family of subexponential growth. On the other hand, to the complete-cut-decomposition $(S,(V_s)_{s \in V(S)})$ of $\Psi$, corresponds a graph-of-groups decomposition of $A(\Psi)$ whose edge-groups are free abelian and whose vertex-groups are the $\langle V_s \rangle$, $s \in V(S)$. The desired conclusion follows from \cite[Proposition~1.13]{bensaid2024coarse}. 
\end{proof}

\noindent
As a concrete application of Corollary~\ref{cor:RAAGJSJ}, let us mention a case where we can prove that there does not exist a coarse embedding between two given right-angled Artin groups. 

\begin{ex}
Let $\Phi$ and $\Psi$ be the two graphs given by Figure~\ref{EmbRAAG}. Clearly, a complete-cut-decomposition of $\Psi$ amounts to describing $\Psi$ as an amalgam between an octogon $\Omega$ and a triangle $\Delta$. The graph $\Phi$ does not contain a complete cut. As a consequence of Corollary~\ref{cor:RAAGJSJ}, if $A(\Phi)$ coarsely embeds into $A(\Psi)$, then $A(\Phi)$ must coarsely embed into either $\langle \Omega \rangle$ or $\langle \Delta \rangle$. But $A(\Phi)$ cannot coarsely embed into $\langle \Omega \rangle$ because it has a larger asymptotic dimension, and it cannot coarsely embed into $\langle \Delta \rangle$ because it has exponential growth. We conclude that $A(\Phi)$ does not coarsely embed into $A(\Psi)$. 
\end{ex}
\begin{figure}
\begin{center}
\includegraphics[width=0.9\linewidth]{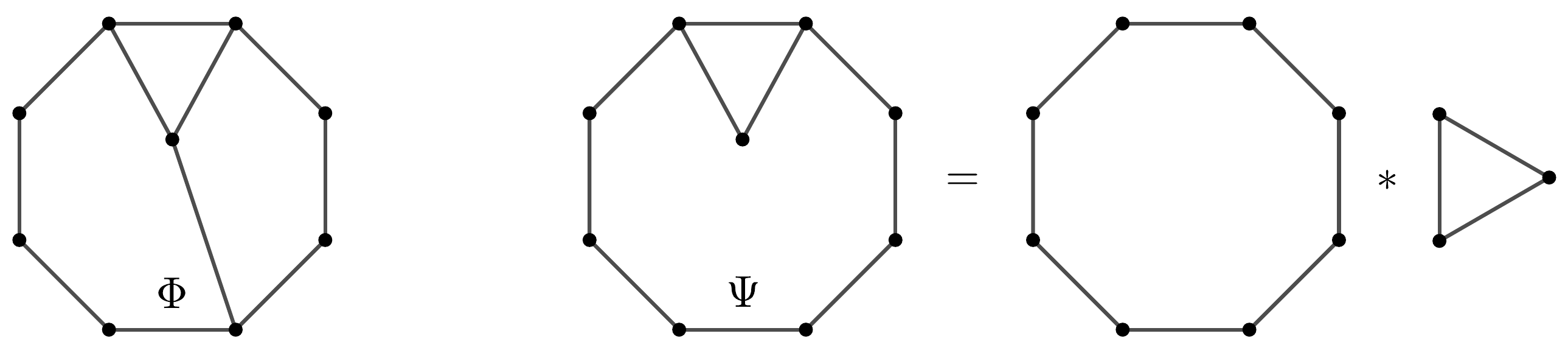}
\caption{There does not exist a coarse embedding $A(\Phi) \to A(\Psi)$.}
\label{EmbRAAG}
\end{center}
\end{figure}

\subsection{Thickness again and non-separability of other classes of groups}\label{sec:thickAgain}

\noindent
We showed in \cite[Theorem 3.5]{bensaid2024coarse} that if a metric space $X$ admits a persistent family with $\delta$-cuts of sizes growing exponentially, then $X \in \mathcal{M}_{exp}$ and for any coarsely-connected metric space $Y$, $X \times Y \in \mathfrak{M}_{exp}$. We improve this for groups in $\mathfrak{M}_{exp}$ that admit a coarse embeddeding of $T_3$. 

\begin{cor}\label{cor:stability_Mexp_product}
Let $X$ be a homogenous metric space in $\mathfrak{M}_{exp}$. If there exists a coarse embedding $T_3 \to X$, then for any coarsely-connected metric space $Y$, $X \times Y \in \mathfrak{M}_{exp}$.
\end{cor}

\begin{proof}
Since coarsely connected metric spaces are quasi-isometric to connected graphs, we may assume that $Y$ is a connected graph. Let $\mathcal{A}$ be the collection of subgraphs $X \times B_Y(y,1)$, and let $\mathcal{B}$ be the collection of subgraphs $X \times \{y\}$, as in \Cref{ex:thick_spaces}. Then $X \times Y$ is $(\mathcal{A},\mathcal{B})$-thick.

\medskip \noindent
    Suppose that $X \times Y$ is coarsely separated by a family $\mathcal{Z}$ of subexponential growth. Then, by \Cref{prop:Thick}, $\mathcal{Z}$ coarsely separates $\mathcal{A}$, since elements of $\mathcal{B}$ have exponential growth. By the homogeneity of $X$ and the existence of a coarsely embedded copy of $T_3$, we may apply \Cref{prop:coarse_sep_family_converse}. It follows that there exists $Z \in \mathcal{Z}$ that coarsely separates a copy $X \times B_Y(y,1)$ for some $y \in Y$. Since $X \times B_Y(y,1)$ is quasi-isometric to $X$, we deduce that $X$ is coarsely separated by a subset of subexponential growth, contradicting the fact that $X \in \mathfrak{M}_{exp}$.
\end{proof}

\begin{cor}\label{cor:stability_Mexp_leftcosets}
    Let $G$ be a locally compact compactly generated group, $S$ a compact generating set, and let $H$ be a closed compactly generated subgroup such that $H \in \mathfrak{M}_{exp}$ and such that there exists a coarse embedding $T_3 \to H$. 
    
    If, for every $s \in S$, the subgroup $H \cap sHs^{-1}$ has exponential relative growth in $G$, then $G \in \mathfrak{M}_{exp}$. In particular, if $H$ is normal, then $G \in \mathfrak{M}_{exp}$.
\end{cor}
\begin{proof}
    Let $\mathcal A$ and $\mathcal B$ be the families as in \Cref{lem:thick_family_left_cosets}. Then $G$ is $(\mathcal{A},\mathcal{B})$-thick. If a family $\mathcal Z$ of subexponential growth coarsely separates $G$, then, by \Cref{prop:Thick} and \Cref{prop:coarse_sep_family_converse}, it coarsely separates the $1$-neighborhood of some left coset of $H$. Since $H$ is coarsely embedded in $G$, the family $\mathcal Z$ still has subexponential growth with respect to the metric of $H$. This contradicts the fact that $H \in \mathfrak{M}_{exp}$.
\end{proof}

\noindent
We now turn to applications of the previous statements to concrete classes of examples.

\paragraph{Free products with amalgamation.} 

\begin{cor}\label{cor:coarse_sep_amalgams}
    Let $G = A_1 *_B A_2$. Suppose $A_1,A_2 \in \mathfrak{M}_{exp}$ and that there exist coarse embeddings $T_3 \to A_1$ and $T_3 \to A_2$. Then $G \in \mathfrak{M}_{exp}$ if and only if $B$ has exponential relative growth.
\end{cor}
\begin{proof}
The ``if'' direction follows from \cite[Proposition 1.13]{bensaid2024coarse}.
    Let $X$ be a Cayley graph of $G$ with respect to a finite generating set. By enlarging the generating set, we can assume that the left cosets of $A_1$, $A_2$, and $B$ are connected subgraphs. Let $\mathcal{A}$ be the family of left cosets of $A_1$ and $A_2$, and let $\mathcal{B}$ be the family of left cosets of $B$. Suppose that $X$ is coarsely separated by a family $\mathcal Z$ of subexponential growth. Since $X$ is $(\mathcal{A},\mathcal{B})$-thick, so by \Cref{prop:Thick}, either there exists $Z \in \mathcal Z$ that coarsely contains some left coset of $B$, and we get a contradiction, or $\mathcal Z$ coarsely separates the family $\mathcal{A}$. Suppose that the latter holds. By \Cref{prop:coarse_sep_family_converse}, $\mathcal Z$ coarsely separates some left coset $g \cdot A_i$ for $i \in \{1,2\}$. Since finitely generated subgroups are coarsely embedded in the ambient group, $\mathcal Z$ still has subexponential growth with respect to the subgroup metric of $A_i$. This contradicts the fact that $A_i \in \mathfrak{M}_{exp}$.
\end{proof}

\paragraph{A class of metabelian connected Lie groups.}

\begin{prop}
Consider a unimodular connected solvable Lie group of the form $G=\mathbb R^d\rtimes \mathbb R^k$, where $\mathbb R^k$ acts on $\mathbb R^d$ by diagonal matrices with positive entries. If $G$ has exponential growth, then $G\in  \mathfrak{M}_{exp}$.
\end{prop}
\begin{proof}
The fact that $G$ has exponential growth implies that there exists an element $a\in \mathbb R^k$ with an eigenvalue $\lambda>1$. Since $G$ is unimodular, the product of the eigenvalues of $a$ equals $1$, and therefore $a$ must also have an eigenvalue $\mu<1$.
It follows that $G$ contains a normal subgroup of the form $\mathbb R^2\rtimes \mathbb R$, where $\mathbb R$ acts via $\mathrm{diag}(\lambda^t,\mu^t)$. Endowed with a suitable left-invariant Riemannian metric, this subgroup $H$ can be viewed as a horocyclic product of two copies of $\mathbb H^2_{\mathbb R}$ (see for instance \cite[\S 3.2]{hume2022poincare}). In particular, $H$ contains a quasi-isometrically embedded copy of $\mathbb{H}^2_{\mathbb R}$, and hence of $T_3$. Therefore, by \cite[Theorem 1.4.]{bensaid2024coarse}, we have $H \in \mathfrak{M}_{exp}$. The conclusion now follows from \ref{cor:stability_Mexp_leftcosets}.
\end{proof}

\noindent

\paragraph{Mapping class groups.} 

\begin{thm}
Let $\Sigma$ be a closed surface of genus $g \geq 4$. Then $\mathrm{MCG}(\Sigma)\in \mathfrak{M}_{exp}$.
\end{thm}

\begin{proof}[Sketch of proof.]
When $g\geq 4$, we show that $G=\mathrm{MCG}(\Sigma)$ is $(\mathcal{A}, \mathcal{B})$-thick where 
$$\mathcal{A}:= \{ \mathrm{MCG} (\Sigma_1) \times \mathrm{MCG}(\Sigma_2) \mid \Sigma = \Sigma_1 \cup \Sigma_2, \ \Sigma_1,\Sigma_2 \text{ disjoint, connected of genus } \geq 1 \},$$ and $\mathcal{B}$ consists of factors of elements of $\mathcal{A}$.

\medskip \noindent
Let $a_1,\ldots, a_g$, $b_1,\ldots, b_g$ and $c_1,\ldots,c_{g-1}$ be the Lickorish generators of $G$. Recall that those are Dehn twists about simple curves $\alpha_i,\beta_i$ and $\gamma_i$, with the only intersecting pairs being $\{\alpha_i,\beta_i\}$, $\{\alpha_i,\gamma_i\}$ and $\{\alpha_{i+1},\gamma_i\}$. 

\medskip \noindent
Consider three such generators $s_1,s_2,s_3$, associated to the curves $\gamma_1$, $\gamma_2$, and $\gamma_3$.
Since $g\geq 4$, one can find a decompositions $\Sigma=\Sigma_1\sqcup \Sigma_2$, and $\Sigma=\Sigma_2\sqcup \Sigma_3$ where each subsurface $\Sigma_i$ are connected and have genius at least $1$, and such that $s_i$ is supported in $\Sigma_i$ (and not along its boundary). This can be verified by a straightforward case-by-case analysis (and the concrete definition of Lickorish generators). In particular, $s_1s_2$ and $s_2s_3$ belong to subgroups $G_1,G_2\in \mathcal{A}$, respectively, with $G_1\cap G_2\in \mathcal{B}$. Moreover, since $s_1\in G_1$, it follows that \[s_1G_1\cap s_1G_2=s_1(G_1\cap G_2),\] which is a coset of a group from $\mathcal{B}$.

\medskip \noindent
It follows that, for any pair of elements $x,y \in G$, one can connect them by a path in the Cayley graph with respect to the Lickorish generators such that any two consecutive vertices $x_i,x_{i+1}$ lie in a coset of some  subgroup $G_i\in \mathcal{A}$, and the intersection of two consecutive such cosets is a coset of a subgroup in $\mathcal{B}$. Hence, $G=\mathrm{MCG}(\Sigma)$ is $(\mathcal{A}, \mathcal{B})$-thick.

\medskip \noindent
To conclude using Corollary~\ref{cor:stability_Mexp_product}, it remains to note that the subgroups $\mathrm{MCG}(\Sigma_i)$ are undistorted (see for instance \cite[Corollary~5.3]{MR2264130}) and contain non-abelian free subgroups (which implies that the $\mathrm{MCG}(\Sigma_1) \times \mathrm{MCG}(G_2)$ are not coarsely separable by families of subexponential growth).
\end{proof}

\addcontentsline{toc}{section}{References}

\bibliographystyle{alpha}
{\footnotesize\bibliography{CoarseSepRAAGs}}

\Address

%

\end{document}